\documentclass{amsart}
\usepackage{amsfonts}

\newtheorem{theorem}{Theorem}[section]
\newtheorem{corollary}[theorem]{Corollary}
\newtheorem{lemma}[theorem]{Lemma}
\newtheorem{definition}[theorem]{Definition}
\newtheorem{proposition}[theorem]{Proposition}
\newtheorem{remark}[theorem]{Remark}
\newtheorem{example}[theorem]{Example}
\numberwithin{theorem}{section}
\input{tcilatex}

\begin{document}
\title[ Ternary domains of completely positive maps on Hilbert $C^{\ast }$%
-modules]{On the ternary domain of a completely positive map on a Hilbert $%
C^{\ast }$-module}
\author[M. B. Asadi]{Mohammad B. Asadi}
\address{School of Mathematics, Statistics and Computer Science, College of
Science, University of Tehran, Tehran, Iran, and \\
School of Mathematics, Institute for Research in Fundamental Sciences (IPM),
P.O. Box: 19395-5746, Tehran, Iran}
\email{mb.asadi@khayam.ut.ac.ir}
\author[R. Behmani]{Reza Behmani}
\address{ Department of Mathematics, Kharazmi University, 50, Taleghani
Ave.,15618, Tehran Iran}
\email{reza.behmani@gmail.com}
\author[M. Joi\c{t}a]{Maria Joi\c{t}a}
\address{Department of Mathematics\\
Faculty of Applied Sciences, University Politehnica of Bucharest, 313 Spl.
Independentei, 060042, Bucharest, Romania}
\email{mjoita@fmi.unibuc.ro and maria.joita@mathem.pub.ro}
\urladdr{http://sites.google.com/a/g.unibuc.ro/maria-joita/}
\subjclass[2010]{ Primary 46L08; Secondary 46L07}
\keywords{Hilbert $C^{\ast }$-modules; completely positive liear maps;
multiplicative domains; ternary maps}

\begin{abstract}
We associate to an operator valued completely positive linear map $\varphi \ 
$on a $C^{\ast }$-algebra $A$ and a Hilbert $C^{\ast }$-module $X$ over $A$
a subset $X_{\varphi }$ of $X,$ called `\textit{ternary domain}' of $\varphi 
$ on $X,$ which is a Hilbert $C^{\ast }$-module over the multiplicative
domain of $\varphi $ and every $\varphi $-map (i.e., associated quaternary
map with $\varphi $) acts on it as a ternary map. We also provide several
characterizations for this set. The ternary domain \ of $\varphi $ on $A\ $%
is a closed two-sided $\ast $-ideal $T_{\varphi }$ of the multiplicative
domain of $\varphi $. We show that $XT_{\varphi }=X_{\varphi }$ and give
several characterizations of the set $X_{\varphi }.$ Furthermore, we
establish some relationships between $X_{\varphi }$ and minimal Stinespring
dilation triples associate to $\varphi $. Finally, we show that every
operator valued completely positive linear map $\varphi $ on a $C^{\ast }$%
-algebra $A$ induces a unique (in a some sense) completely positive linear
map on the linking algebra of $X$ and we determine its multiplicative domain
in terms of the multiplicative domain of $\varphi $ and the ternary domain
of $\varphi $ on $X$.
\end{abstract}

\maketitle

\section{Introduction}

Hilbert $C^{\ast }$-modules now play an important role in the wide area of
the theory of $C^{\ast }$-algebras such as $KK$-theory, non-commutative
geometry, Morita equivalence of $C^{\ast }$-algebras and else. Since each $%
C^{\ast }$-algebra is a Hilbert $C^{\ast }$-module, the notion of Hilbert $%
C^{\ast }$-module can be regarded as a generalization of the notion of $%
C^{\ast }$-algebra. In view of all this, the most of well-known notions and
results in the category of $C^{\ast }$-algebras can be investigated in the
category of Hilbert $C^{\ast }$-modules. For instance, some authors studied
the dilation theory for $\varphi $-maps on Hilbert $C^{\ast }$-modules as a
counterpart of the dilation of completely positive linear maps on $C^{\ast }$%
-algebras to $\ast $-representations (see \cite{Asadi, ABN1, ABN2, BRS,
J2011, J2012, SS}).

The multiplicative domain of a completely positive linear map $\varphi $ has
a closed correlation to the dilation of $\varphi $. Indeed, the
multiplicative domain of a completely positive linear map $\varphi $ is the
largest $C^{\ast }$-subalgebra of the domain of $\varphi $ that the
restriction of $\varphi $ to it is a $\ast $-homomorphism.

In this paper, we introduce a notion of \textit{$\varphi $-module domain}
for $\varphi $-maps on Hilbert $C^{\ast }$-modules, and provide several
characterizations of this notion (see Proposition \ref{characterization},
Theorem \ref{st}, Corollaries \ref{cor}, \ref{st10}, \ref{Phich}, \ref{stch}
and Lemma \ref{lem42}). As a motivation, we note that when $X$ is a Hilbert $%
C^{\ast }$-module over a $C^{\ast }$-algebra $A$ and $\Phi :X\rightarrow L(%
\mathcal{H},\mathcal{K})$ is a $\varphi $-representation (i. e., $\varphi
:A\rightarrow L(\mathcal{H})$ is a $\ast $-homomorphism), then $\Phi $ is a $%
\varphi $-module map, i. e., $\Phi (xa)=\Phi (x)\varphi (a)$ for all $x\in X$
and $a\in A$. However, if $\varphi $ is only a completely positive linear
map, then $\Phi $ is not necessarily a $\varphi $-module map.

We show that when $\varphi $ is a completely positive linear map, then for a 
$\varphi $-map $\Phi $ on a Hilbert $C^{\ast }$-module $X$, the $\varphi $%
-module domain of $\Phi $ (which is denoted by $X_{\Phi }$) is a Hilbert $%
C^{\ast }$-module over $M_{\varphi }$ (the multiplicative domain of $\varphi 
$) and the restriction of $\Phi $ to $X_{\Phi }$ is a $\varphi $-module map.
As an interesting result, we prove that $X_{\Phi }$ is independent of $%
\varphi $-map $\Phi $ and only dependents on $\varphi $ and $X$. Hence, we
can denote $X_{\Phi }$ by $X_{\varphi }$ and call it the `\textit{ternary
domain}' of the completely positive linear map $\varphi $ on $X$. Since
every $C^{\ast }$-algebra is a Hilbert $C^{\ast }$-module over itself, we
define the ternary domain of $\varphi $ on $A$ and denote it by $T_{\varphi
} $. We show that $T_{\varphi }$ has a closed relation with the ternary
domain of $\varphi $ on a Hilbert $A$-module $X$ by showing that $%
XT_{\varphi }=X_{\varphi }$ (see Proposition \ref{tttheorem}).

Finally, we show that each $\varphi $-map $\Phi :X\rightarrow L(\mathcal{H},%
\mathcal{K})$ induces a unique ( in a some sense) completely positive linear
map on the linking $C^{\ast }$-algebra of $X,\ \widetilde{\varphi }_{\Phi }:%
\mathcal{L}(X)\rightarrow L(\mathcal{K}\oplus \mathcal{H})$ such that $%
\mathcal{L}_{M_{\varphi }}(X_{\varphi })$, the linking $C^{\ast }$-algebra
of the ternary domain of $\varphi $, can be identify with a $C^{\ast }$%
-subalgebra of $\ M_{\widetilde{\varphi }_{\Phi }}$, the multiplicative
domain of $\ \widetilde{\varphi }_{\Phi }$. Moreover, we show that all the
induced completely positive linear maps on the linking $C^{\ast }$-algebra
of a Hilbert $C^{\ast }$-module $X\ $have the same multiplicative domain and
ternary domain. In other word, for any $\varphi $-map $\Phi :X\rightarrow L(%
\mathcal{H},\mathcal{K^{\prime }})$, the multiplicative domain $M_{%
\widetilde{\varphi }_{\Phi }}$ and the ternary domain $T_{\widetilde{\varphi 
}_{\Phi }}$ are unique objects which are independent of $\Phi $ and only
depend on $\varphi $ and $X$.


\section{Preliminaries}

Let $A$ and $B$ be two $C^{\ast }$-algebras. A linear map $\varphi
:A\rightarrow B$ is \textit{positive} if $\varphi \left( a^{\ast }a\right)
\geq 0\ $ for all $a\in A$. For each positive integer $n,$ $M_{n}(A)$
denotes the $C^{\ast }$-algebra of all $n\times n$ matrices with elements in 
$A$. A positive linear map $\varphi :A\rightarrow B$ is a \textit{completely
positive linear map} if for each positive integer $n$, the linear map $%
\varphi _{n}:M_{n}(A)\rightarrow M_{n}(B),$%
\begin{equation*}
\varphi _{n}\left( \left[ a_{ij}\right] _{i,j=1}^{n}\right) =\left[ \varphi
\left( a_{ij}\right) \right] _{i,j=1}^{n}
\end{equation*}%
is positive.

The \textit{multiplicative domain} of the completely positive linear map $%
\varphi :A\rightarrow B,$ 
\begin{equation*}
M_{\varphi }=\{a\in A:\varphi \left( ab\right) =\varphi \left( a\right)
\varphi \left( b\right) \text{ and }\varphi \left( ba\right) =\varphi \left(
b\right) \varphi \left( a\right) \text{ for all }b\in A\}
\end{equation*}%
is a $C^{\ast }$-subalgebra of $A$. Clearly, if $\varphi $ is unital, then $%
1_{A}$, the unit of $A$, is an element of $M_{\varphi }$. In addition, if $%
\varphi $ is contractive, then 
\begin{equation*}
M_{\varphi }=\{a\in A:\varphi \left( aa^{\ast }\right) =\varphi \left(
a\right) \varphi \left( a\right) ^{\ast }\text{ and }\varphi \left( a^{\ast
}a\right) =\varphi \left( a\right) ^{\ast }\varphi \left( a\right) \}
\end{equation*}%
and so $M_{\varphi }$ is the largest $C^{\ast }$-subalgebra $C$ of $A$ such
that the map $\varphi |_{C}:C\rightarrow B$ (the restriction of $\varphi $
to $C$) is a $C^{\ast }$-morphism (see, \cite[Theorem 3.18]{P} and \cite[%
Theorem 2.1]{G}).

Stinespring \cite{St} showed that a completely positive linear map $\varphi
:A\rightarrow L(\mathcal{H})$, where $L(\mathcal{H})$ denotes the $C^{\ast }$%
-algebra of all bounded linear operators on a Hilbert space $\mathcal{H}$,$\ 
$is of the form $\varphi \left( \cdot \right) =V_{\varphi }^{\ast }\pi
_{\varphi }\left( \cdot \right) V_{\varphi }$, where $\pi _{\varphi }$ is a $%
\ast $-representation of $A$ on a Hilbert space $\mathcal{H}_{\varphi }$ and 
$V_{\varphi }$ is a bounded linear operator from $\mathcal{H}$ to $\mathcal{H%
}_{\varphi }$. The triple $\left( \pi _{\varphi },\mathcal{H}_{\varphi
},V_{\varphi }\right) $ is called a Stinespring representation associated to 
$\varphi $. If $A$ is a unital $C^{\ast }$-algebra and $\varphi $ is unital,
then $V_{\varphi }$ is an isometry. If $\mathcal{H}_{\varphi }=\left[ \pi
_{\varphi }\left( A\right) V_{\varphi }\mathcal{H}\right] $, the Hilbert
space generated by $\{\pi _{\varphi }\left( a\right) V_{\varphi }\xi :a\in
A,\xi \in \mathcal{H}\}$, the triple $\left( \pi _{\varphi },\mathcal{H}
_{\varphi },V_{\varphi }\right) $ is called a minimal Stinespring
representation associated to $\varphi $. If $\left( \pi,\mathcal{K},V\right) 
$ is another minimal Stinespring representation associated to $\varphi $,
then there is a unique unitary operator $U:\mathcal{H}_{\varphi }\rightarrow 
\mathcal{K}$ such that $UV_{\varphi }=V$ and $\pi _{\varphi }(\cdot
)=U^{\ast }\pi (\cdot )U$. Therefore, two minimal Stinespring
representations are unitarily equivalent.

A \textit{Hilbert }$C^{\ast }$\textit{-module} $X$ over a $C^{\ast }$%
-algebra $A$ is by definition a linear space $X$ that is a right $A$-module,
together with an $A$-valued inner product $\left\langle \cdot ,\cdot
\right\rangle $ that is $A$-linear in the second variable and conjugate
linear in the first variable, such that $X$ is a Banach space with the norm
induced by the inner product, $\left\Vert x\right\Vert =\left\Vert
\left\langle x,x\right\rangle \right\Vert ^{\frac{1}{2}}$. We say that $X$
is \textit{full} if $A$ coincides with the closed two-sided ideal generated
by $\{\left\langle x,y\right\rangle : x,y\in X\}.$

We denote the $C^{\ast }$-algebra of all `compact' operators on a Hilbert $%
C^{\ast }$-module $X$ by $K(X)$. We also use $K(X,Y)$ to denote space of all
`compact' operators acting between different Hilbert $C^{\ast }$-modules $X$
and\ $Y$ over $A$. Moreover, $K(X,Y)$ is generated by $\{\theta _{y,x}:x\in
X,y\in Y\}$, where $\theta _{y,x}:X\rightarrow Y,\theta _{y,x}\left(
z\right) =y\left\langle x,z\right\rangle .$

For a Hilbert $C^{\ast }$-module $X$ over a $C^{\ast }$-algebra $A$, the
linking $C^{\ast }$-algebra of $X$ is denoted by $\mathcal{L}(X)$ and
defined as%
\begin{equation*}
\mathcal{L}(X)=\left[ 
\begin{array}{cc}
K(X) & K(A,X) \\ 
K(X,A) & K(A)%
\end{array}%
\right] .
\end{equation*}%
We remark that $\mathcal{L}(X)$ is in fact $C^{\ast }$-algebra of all 
\textit{`compact'} operators on the Hilbert $C^{\ast }$-module $X$ $\oplus A$
over $A$. The map $x\mapsto l_{x},$ where $l_{x}:X\rightarrow
A,l_{x}(z)=\left\langle x,z\right\rangle $, is an isometric conjugate linear
isomorphism from $X$ to $K(X,A)$, and we denote $K(X,A)$ by $X^{\ast }$ and $%
l_{x}$ by $x^{\ast }$. The map $x\mapsto r_{x},$ where $r_{x}:A\rightarrow
X,r_{x}(a)=xa\ (r_{x}^{\ast }=l_{x})$, is an isometric linear isomorphism
from $X$ to $K(A,X)$, and we denote $K(A,X)$ by $X$ and $r_{x}$ by $x$. The
map $a\mapsto T_{a}$, where $T_{a}:A\rightarrow A,$ $T_{a}\left( b\right)
=ab $, is an isometric linear $\ast $-isomorphism from $A$ to $K(A)$, and we
denote $K(A)$ by $A$ and $T_{a}$ by $a$. Using the mentioned isomorphisms,%
\begin{equation*}
\mathcal{L}(X)=\left[ 
\begin{array}{cc}
K(X) & X \\ 
X^{\ast } & A%
\end{array}%
\right] =\left\{ \left[ 
\begin{array}{cc}
T & x \\ 
y^{\ast } & a%
\end{array}%
\right] :T\in K(X),x,y\in X,a\in A\right\} .
\end{equation*}

Let $X$ be a Hilbert $C^{\ast }$-module over $A$ and $\varphi :A\rightarrow
L(\mathcal{H})$ a linear map. A map $\Phi :X\rightarrow L(\mathcal{H},%
\mathcal{K})$ is called $\varphi $\textit{-map} if 
\begin{equation*}
\Phi \left( x\right) ^{\ast }\Phi \left( y\right) =\varphi \left(
\left\langle x,y\right\rangle \right)
\end{equation*}%
for all $x,y\in X$. We say that the $\varphi $-map $\Phi $ is\textit{\
non-degenerate} if $\left[ \Phi \left( X\right) \mathcal{H}\right] =\mathcal{%
K}$ and $\left[ \Phi \left( X\right) ^{\ast }\mathcal{K}\right] =\mathcal{H}$%
. If $X$ is full and $\varphi $ is a $C^{\ast }$-morphism, then $\Phi $ is
non-degenerate if and only if $\left[ \Phi \left( X\right) \mathcal{H}\right]
=\mathcal{K}$. If $\varphi $ is a $C^{\ast }$-morphism we say that $\Phi \ $%
is a $\varphi $-representaion of $X$ on the Hilbert spaces $\mathcal{H}$ and 
$\mathcal{K}$, and if $\varphi $ is a completely positive linear map we say
that $\Phi \ $is a $\varphi $-completely positive linear map.

Let $X$ be a Hilbert $C^{\ast }$-module over the $C^{\ast }$-algebra $A$, $%
\varphi :A\rightarrow L(\mathcal{H})$ a completely positive linear map and $%
\Phi :X\rightarrow L(\mathcal{H},\mathcal{K})$ a $\varphi $-map. Bhat,
Ramesh and Sumesh \cite[Theorem 2.1]{BRS} showed that there is a triple of
pairs $(\left( \Pi _{\Phi },\pi _{\varphi }\right) ,(W_{\Phi },V_{\varphi
}), $ $(\mathcal{H}_{\varphi },\mathcal{K}_{\Phi }))$ consisting of the
Hilbert spaces $\mathcal{H}_{\varphi }$ and $\mathcal{K}_{\Phi }$, a bounded
linear operator $V_{\varphi }:\mathcal{H}\rightarrow \mathcal{H}_{\varphi }$%
, a coisometry $W_{\Phi }:\mathcal{K}\rightarrow \mathcal{K}_{\Phi }$, a $%
\ast $-representation $\pi _{\varphi }:A\rightarrow L(\mathcal{H}_{\varphi
}) $ and a $\pi _{\varphi }$-representation $\Pi _{\Phi }:X\rightarrow L(%
\mathcal{H}_{\varphi },\mathcal{K}_{\Phi })$ such that $\left( \pi _{\varphi
},\mathcal{H}_{\varphi },V_{\varphi }\right) $ is a minimal Stinespring
representation associated to $\varphi $ and $\Phi \left( \cdot \right)
=W_{\Phi }^{\ast }\Pi _{\Phi }\left( \cdot \right) V_{\varphi }$. Moreover, $%
\left[ \Pi _{\Phi }\left( X\right) V_{\varphi }\mathcal{H}\right] =\mathcal{K%
}_{\Phi }$. The triple of pairs $(\left( \Pi _{\Phi },\pi _{\varphi }\right)
,(W_{\Phi },V_{\varphi }),$ $(\mathcal{H}_{\varphi },\mathcal{K}_{\Phi }))$
is called a minimal Stinespring representation associated to the pair $%
(\varphi , \Phi )$.

Let $X$ be a Hilbert $C^{\ast }$-module over a $C^{\ast }$-algebra $A$. The 
\textit{ternary product} on $X$ is the map $[\cdot ,\cdot ,\cdot ]:X\times
X\times X\rightarrow X$ defined by 
\begin{equation*}
\lbrack x,y,z]=x\langle y,z\rangle ,\forall x,y,z\in X.
\end{equation*}%
A map between two Hilbert $C^{\ast }$-modules is called a \textit{ternary map%
}, if it preserves the ternary product. In the case of operator-valued maps,
a map $\Phi :X\rightarrow L(\mathcal{H},\mathcal{K})$ is a ternary map, if 
\begin{equation*}
\Phi (x\langle y,z\rangle )=\Phi (x)\Phi (y)^{\ast }\Phi (z)
\end{equation*}%
for all $x,y,z\in X$. It is known that if $X$ and $Y$ are two Hilbert $%
C^{\ast }$-modules over the $C^{\ast }$-algebras $A$ and $B$, respectively, $%
\varphi :A\rightarrow B$ is a $\ast $-homomorphism and $\Phi :X\rightarrow Y$
is a $\varphi $-map, then $\Phi $ is a linear ternary map. The converse is
also true, that is, the linear ternary maps between full Hilbert $C^{\ast }$%
-modules are exactly those $\varphi $-maps which $\varphi $ is a $\ast $%
-homomorphism between the underlying $C^{\ast }$-algebras of the Hilbert $%
C^{\ast }$-modules.


\section{Ternary domains}

In \cite{SS}, Skeide and Sumesh introduced the notion of quaternary map
between Hilbert $C^{\ast }$-modules as a generalization of the notion of
ternary map and characterized $\varphi $-maps as completely bounded
quaternary maps. In this section, we introduce the notion of a ternary
domain of an operator-valued completely bounded quaternary map and we show
that all the completely bounded quaternary maps which associate with the
same completely positive linear map have the same module domain and acts on
it as a ternary map.


Throughout the paper, $A$ is a $C^{\ast }$-algebra and $X$ a Hilbert $%
C^{\ast }$-module over $A$. Let $\varphi :A\rightarrow L(\mathcal{H})$ be a
completely positive linear map and $\Phi :X\rightarrow L(\mathcal{H},%
\mathcal{K})$ be a $\varphi $-map. We set 
\begin{equation*}
X_{\Phi }=\{x\in X:\Phi \left( xb\right) =\Phi \left( x\right) \varphi
\left( b\right) \text{ for all }b\in A\}.
\end{equation*}%
\ The map $\Phi $ acts on $X_{\Phi }$ as a $\varphi $-module map and we call
it the $\varphi $\textit{-module domain} of $\Phi $. Also, we set 
\begin{equation*}
T_{\Phi }=\{y\in X:\Phi \left( x\left\langle y,z\right\rangle \right) =\Phi
\left( x\right) \Phi \left( y\right) ^{\ast }\Phi \left( z\right) \text{ for
all }x,z\in X\}.
\end{equation*}%
Clearly, $\Phi $ acts on $T_{\Phi }$ as a ternary map and we call it the 
\textit{ternary domain} of $\Phi $. We will show that $T_{\Phi }=X_{\Phi }$.

Also, we define the \textit{ternary domain of $\varphi $} as follows:%
\begin{equation*}
T_{\varphi }=\{a\in M_{\varphi }:\varphi (ba^{\ast }c)=\varphi (b)\varphi
(a)^{\ast }\varphi (c)\text{ for all }b,c\in A\}.
\end{equation*}%
We note that $T_{\varphi }$ is a closed two-sided $\ast $-ideal in $%
M_{\varphi }$ and 
\begin{equation*}
T_{\varphi }=\{a\in A:\varphi (bac)=\varphi (ba)\varphi (c)=\varphi
(b)\varphi (ac)=\varphi (b)\varphi (a)\varphi (c)\text{ for all }b,c\in A\}.
\end{equation*}%
In particular, $T_{\varphi }AT_{\varphi }=T_{\varphi }$ and if $\varphi $ is
unital (or $\varphi (A)$ has a unital element in $\varphi (M_{\varphi }))$,
then 
\begin{equation*}
T_{\varphi }=\{a\in A:\varphi (bac)=\varphi (b)\varphi (a)\varphi (c)\text{
for all }b,c\in A\}.
\end{equation*}


\begin{example}
\label{example0} The map $\Phi :M_{2}(\mathbb{C})\rightarrow L(\mathbb{C}%
^{2},\mathbb{C}^{2})$ defined by 
\begin{equation*}
\Phi \left( \left[ 
\begin{array}{cc}
x_{11} & x_{12} \\ 
x_{21} & x_{22}%
\end{array}%
\right] \right) =\left[ 
\begin{array}{cc}
x_{11} & 0 \\ 
x_{21} & 0%
\end{array}%
\right]
\end{equation*}%
is a $\varphi $-map, where $\varphi :M_{2}(\mathbb{C})\rightarrow M_{2}(%
\mathbb{C})\ $is given by 
\begin{equation*}
\varphi \left( \left[ 
\begin{array}{cc}
a_{11} & a_{12} \\ 
a_{21} & a_{22}%
\end{array}%
\right] \right) =\left[ 
\begin{array}{cc}
a_{11} & 0 \\ 
0 & 0%
\end{array}%
\right] .
\end{equation*}%
It is easy to check that $\varphi $ is a completely positive linear map, $%
M_{\varphi }=\left\{ \left[ 
\begin{array}{cc}
a & 0 \\ 
0 & b%
\end{array}%
\right] :a,b\in \mathbb{C}\right\} $, $X_{\Phi }=\left\{ \left[ 
\begin{array}{cc}
x & 0 \\ 
y & 0%
\end{array}%
\right] :x,y\in \mathbb{C}\right\} =T_{\Phi }$ and $T_{\varphi }=\left\{ %
\left[ 
\begin{array}{cc}
a & 0 \\ 
0 & 0%
\end{array}%
\right] :a\in \mathbb{C}\right\} .$
\end{example}

\begin{example}
Let $A$ be a $C^*$-algebra and $\varphi :A\rightarrow L(\mathcal{H})$ a
completely positive linear map and $\Phi : A\rightarrow L(\mathcal{H},%
\mathcal{K})$ be a $\varphi $-map. Then

\begin{enumerate}
\item $A_{\Phi }\subseteq \{a\in A:\varphi \left( ab\right) =\varphi \left(
a\right) \varphi \left( b\right) \text{ for all }b\in A\}$.

Indeed, let $a\in A_{\Phi }$, $b\in A$ and $(e_{\alpha})_{\alpha \in I}$ be
an approximate unit for $A$. Then 
\begin{eqnarray*}
\varphi \left( ab\right) &=& \lim_{\alpha} \varphi \left( \left\langle
e_{\alpha},ab\right\rangle \right) = \lim_{\alpha} \Phi \left(
e_{\alpha}\right) ^{\ast }\Phi \left( ab\right) = \lim_{\alpha} \Phi \left(
e_{\alpha}\right) ^{\ast }\Phi \left( a\right) \varphi \left( b\right) \\
&=& \lim_{\alpha} \varphi \left( \left\langle e_{\alpha},a\right\rangle
\right) \varphi \left( b\right) =\varphi \left( a\right) \varphi \left(
b\right).
\end{eqnarray*}

\item If $A$ is unital and $\Phi \left( 1_{A}\right) \ $ is onto, then 
\begin{equation*}
A_{\Phi }=\{a\in A:\varphi \left( ab\right) =\varphi \left( a\right) \varphi
\left( b\right) \text{ for all }b\in A\}.
\end{equation*}

Indeed, if $\Phi \left( 1_{A}\right) \ $is onto, then its adjoint has a left
inverse. Thus for every $a\in A$ with the property that $\varphi \left(
ab\right) =\varphi \left( a\right) \varphi \left( b\right) $ for all $b\in A$%
, we have 
\begin{eqnarray*}
\Phi \left( 1_{A}\right) ^{\ast }\left( \Phi \left( ab\right) -\Phi \left(
a\right) \varphi \left( b\right) \right) &=&\varphi \left( \left\langle
1_{A},ab\right\rangle \right) -\varphi \left( \left\langle
1_{A},a\right\rangle \right) \varphi \left( b\right) \\
&=&\varphi \left( ab\right) -\varphi \left( a\right) \varphi \left( b\right)
=0\text{ for all }b\in A.
\end{eqnarray*}%
From this relation and taking into account that $\Phi \left( 1_{A}\right)
^{\ast }$ has a left inverse, we conclude that $\Phi \left( ab\right) =\Phi
\left( a\right) \varphi \left( b\right) $ for all $b\in A$, and so $a\in $ $%
A_{\Phi }.$

\item If $A$ is unital and $\Phi \left( 1_{A}\right) $ is a coisometry then $%
\varphi $ is a $\ast $-homomorphism and $A_{\Phi }=A$.

Indeed, for every $a\in A$, we have%
\begin{equation*}
\varphi (a)=\varphi \left( \left\langle 1_{A},a\right\rangle \right) =\Phi
(1_{A})^{\ast }\Phi (a),
\end{equation*}%
and then 
\begin{eqnarray*}
\varphi (a)\varphi (b) &=&(\varphi (a^{\ast }))^{\ast }\varphi (b)=\Phi
(a^{\ast })^{\ast }\Phi (1_{A})\Phi (1_{A})^{\ast }\Phi (b) \\
&=&\Phi (a^{\ast })^{\ast }\Phi (b)=\varphi (\left\langle a^{\ast
},b\right\rangle )=\varphi (a b)
\end{eqnarray*}%
for all $a,b\in A$.
\end{enumerate}
\end{example}

In the following theorem, we provide some properties of the $\varphi $%
-module domain of a $\varphi $-map on a Hilbert $C^{\ast }$-module. 

\begin{theorem}
\label{main theorem} Let $X$ be a Hilbert $A$-module, $\varphi :A\rightarrow
L(\mathcal{H})$ a completely positive linear map and $\Phi :X\rightarrow L(%
\mathcal{H},\mathcal{K})$ a $\varphi $-map. Then:

\begin{enumerate}
\item $X_{\Phi }$ is a Hilbert $C^{\ast }$-module over $M_{\varphi }$;

\item $\left. \Phi \right\vert _{X_{\Phi }}:X_{\Phi }\rightarrow L(\mathcal{H%
},\mathcal{K})\ $is a $\left. \varphi \right\vert _{M_{\varphi }}$%
-representation;

\item $\left. \Phi \right\vert _{X_{\Phi }}:X_{\Phi }\rightarrow L(\mathcal{H%
},\mathcal{K})\ $ is a ternary map.
\end{enumerate}
\end{theorem}

\begin{proof}
$(1)$ Let $x,y\in X_{\Phi },$ $a\in M_{\varphi }$. Then 
\begin{equation*}
\Phi \left( \left( x+y\right) b\right) =\Phi \left( xb\right) +\Phi \left(
yb\right) =\Phi \left( x\right) \varphi \left( b\right) +\Phi \left(
y\right) \varphi \left( b\right) =\Phi \left( x+y\right) \varphi \left(
b\right)
\end{equation*}%
for all $b\in A$,$\ $and%
\begin{equation*}
\Phi \left( \left( xa\right) b\right) =\Phi \left( x\left( ab\right) \right)
=\Phi \left( x\right) \varphi \left( ab\right) =\Phi \left( x\right) \varphi
\left( a\right) \varphi \left( b\right) =\Phi \left( xa\right) \varphi
\left( b\right)
\end{equation*}%
for all $b\in A$. Thus, we showed that $X_{\Phi }$ is a right module over
the $C^{\ast }$-algebra $M_{\varphi }$. From 
\begin{equation*}
\varphi \left( \left\langle x,y\right\rangle b\right) =\varphi \left(
\left\langle x,yb\right\rangle \right) =\Phi \left( x\right) ^{\ast }\Phi
\left( y\right) \varphi \left( b\right) =\varphi \left( \left\langle
x,y\right\rangle \right) \varphi \left( b\right)
\end{equation*}%
and 
\begin{equation*}
\varphi \left( b\left\langle x,y\right\rangle \right) =\varphi \left(
\left\langle xb^{\ast },y\right\rangle \right) =\varphi \left( b^{\ast
}\right) ^{\ast }\Phi \left( x\right) ^{\ast }\Phi \left( y\right) =\varphi
\left( b\right) \varphi \left( \left\langle x,y\right\rangle \right)
\end{equation*}%
for all $b\in A,$ we deduce that $\left\langle x,y\right\rangle \in
M_{\varphi },$ and so $X_{\Phi }$ is a pre-Hilbert module over the $C^{\ast
} $-algebra $M_{\varphi }.$

Let $\{x_{i}\}_{i\in I}$ be a net of elements in $X_{\Phi }$ which converges
to $x\in X$\ and $b\in A$. Then, since the net $\{x_{i}b\}_{i\in I}$
converges to $xb$ and $\Phi $ is continuous, 
\begin{equation*}
\Phi \left( xb\right) =\lim\limits_{i}\Phi \left( x_{i}b\right)
=\lim\limits_{i}\Phi \left( x_{i}\right) \varphi \left( b\right) =\Phi
\left( x\right) \varphi \left( b\right)
\end{equation*}%
and so $X_{\Phi }$ is a Hilbert $C^{\ast }$-module over $M_{\varphi }.$

$(2)$ Since $X_{\Phi }$ is a Hilbert $C^{\ast }$-module over $M_{\varphi }$
and $\left. \varphi \right\vert _{M_{\varphi }}$is a $\ast $-representation,$%
\left. \Phi \right\vert _{X_{\Phi }}:X_{\Phi }\rightarrow L(\mathcal{H},%
\mathcal{K})\ $is a $\left. \varphi \right\vert _{M_{\varphi }}$%
-representation.

$(3)$ By part $(2)$, $\left. \Phi \right\vert _{X_{\Phi }}:X_{\Phi
}\rightarrow L(\mathcal{H},\mathcal{K})\ $is a $\left. \varphi \right\vert
_{M_{\varphi }}$-representation and so it is a ternary map.
\end{proof}


\begin{lemma}
\label{idmdlem} Let $I$ be a closed two-sided $\ast $-ideal of $A$ and $%
\varphi :A\rightarrow L(\mathcal{H})$ a positive linear map. Then $M_{\left.
\varphi \right\vert _{I}}\subseteq M_{\varphi }$, where $\left. \varphi
\right\vert _{I}$ is the restriction of $\varphi $ to $I$.
\end{lemma}

\begin{proof}
Let $b\in A,$ $a\in I$ and $(u_{\lambda })_{\lambda \in \Lambda }$ be an
approximate unit for $I$. From 
\begin{eqnarray*}
\left\Vert au_{\lambda }bu_{\lambda }-ab\right\Vert &\leq &\left\Vert \left(
au_{\lambda }-a\right) bu_{\lambda }\right\Vert +\left\Vert abu_{\lambda
}-ab\right\Vert \\
&\leq &\left\Vert au_{\lambda }-a\right\Vert \left\Vert b\right\Vert
+\left\Vert abu_{\lambda }-ab\right\Vert
\end{eqnarray*}%
for all $\lambda \in \Lambda \ $and taking into account that $(u_{\lambda
})_{\lambda \in \Lambda }$ is an approximate unit for $I$, we deduce that%
\begin{equation*}
\lim_{\lambda }\Vert au_{\lambda }bu_{\lambda }-ab\Vert =0.
\end{equation*}

On the other hand, $I$ is a hereditary $C^{\ast }$-subalgebras of $A$, and
then, for each $b\in A,$ we have 
\begin{equation*}
\tau \circ \varphi (b)=\lim_{\lambda }\tau \circ \varphi (u_{\lambda
}bu_{\lambda })
\end{equation*}%
for all positive (and so all) linear functional $\tau $ on $L(\mathcal{H})\ $%
\ (\cite[Theorem 3.3.9]{M}). In particular, we have 
\begin{equation*}
\langle \varphi (b)\xi ,\eta \rangle =\lim_{\lambda }\langle \varphi
(u_{\lambda }bu_{\lambda })\xi ,\eta \rangle
\end{equation*}%
for all $\xi ,\eta \in \mathcal{H}.$ Now, let $a\in M_{\left. \varphi
\right\vert _{I}}\ $and $b\in A$. Then 
\begin{eqnarray*}
\langle \varphi (ab)\xi ,\eta \rangle &=&\lim_{\lambda }\langle \varphi
(au_{\lambda }bu_{\lambda })\xi ,\eta \rangle =\lim_{\lambda }\langle
\varphi (a)\varphi (u_{\lambda }bu_{\lambda })\xi ,\eta \rangle \\
&=&\lim_{\lambda }\langle \varphi (u_{\lambda }bu_{\lambda })\xi ,\varphi
(a)^{\ast }\eta \rangle =\langle \varphi (b)\xi ,\varphi (a)^{\ast }\eta
\rangle =\langle \varphi (a)\varphi (b)\xi ,\eta \rangle
\end{eqnarray*}%
for all $\xi ,\eta \in \mathcal{H},$ and so $\varphi (ab)=\varphi (a)\varphi
(b)$. Similarly, we can show that $\varphi (ba)=\varphi (b)\varphi (a)$.
Therefore, $a\in M_{\varphi }$.
\end{proof}


In the next proposition, a characterization of $\varphi $-module domain of $%
\Phi $ is provided in terms of the ternary domain of $\varphi \ $and the
ternary domain of $\Phi $.


\begin{proposition}
\label{characterization} Let $X$ be a Hilbert $A$-module, $\varphi
:A\rightarrow L(\mathcal{H})$ a completely positive linear map, $\Phi
:X\rightarrow L(\mathcal{H},\mathcal{K})$ a $\varphi $-map and $x_{0}\in X$.
Then

\begin{enumerate}
\item $x_{0}\in X_{\Phi }$ if and only if $\left\langle x_{0} ,
x_{0}\right\rangle \in T_\varphi$;

\item $x_{0}\in X_{\Phi }$ if and only if $\Phi \left( x_{0}\left\langle y ,
z \right\rangle \right) =\Phi \left( x_{0}\right) \varphi \left(
\left\langle y , z\right\rangle \right) $, for all $y, z \in X$;

\item $X_{\Phi }= T_{\Phi }$.\newline
Moreover, if $\varphi$ is contractive, then

\item $\Phi \left( x_{0}\left\langle x_{0} , x_{0}\right\rangle \right)
=\Phi \left( x_{0}\right) \varphi \left( \left\langle x_{0} ,
x_{0}\right\rangle \right) $ if and only if $\left\langle x_{0} ,
x_{0}\right\rangle \in M_\varphi$.
\end{enumerate}
\end{proposition}

\begin{proof}
$(1)$ First we suppose that $x_{0}\in X_{\Phi }$. Since $X_\Phi$ is a
Hilbert $C^*$-module over $M_\varphi$ (by Theorem \ref{main theorem}), $%
\left\langle x_{0},x_{0}\right\rangle \in M_{\varphi }$. Also, 
\begin{eqnarray*}
\varphi \left( b\left\langle x_{0},x_{0}\right\rangle c\right) &=&\varphi
\left( \left\langle x_{0}b^{\ast },x_{0}c\right\rangle \right) =\Phi \left(
x_{0}b^{\ast }\right) ^{\ast }\Phi \left( x_{0}c\right) \\
&=&\varphi \left( b\right) \Phi \left( x_{0}\right) ^{\ast }\Phi \left(
x_{0}\right) \varphi \left( c\right) =\varphi \left( b\right) \varphi \left(
\left\langle x_{0},x_{0}\right\rangle \right) \varphi \left( c\right)
\end{eqnarray*}%
for all $b,c\in A$. Hence, $\left\langle x_{0},x_{0}\right\rangle \in
T_{\varphi }$.

Conversely, let $\left\langle x_{0},x_{0}\right\rangle \in T_{\varphi }$.
Then, for every $b\in A$, we have 
\begin{eqnarray*}
&&\left( \Phi \left( x_{0}b\right) -\Phi \left( x_{0}\right) \varphi \left(
b\right) \right) ^{\ast }\left( \Phi \left( x_{0}b\right) -\Phi \left(
x_{0}\right) \varphi \left( b\right) \right) \\
&=&\left( \Phi \left( x_{0}b\right) ^{\ast }-\varphi \left( b\right) ^{\ast
}\Phi \left( x_{0}\right) ^{\ast }\right) \left( \Phi \left( x_{0}b\right)
-\Phi \left( x_{0}\right) \varphi \left( b\right) \right) \\
&=&\varphi \left( b^{\ast }\left\langle x_{0},x_{0}\right\rangle b\right)
-\varphi \left( b^{\ast }\left\langle x_{0},x_{0}\right\rangle \right)
\varphi \left( b\right) -\varphi \left( b\right) ^{\ast }\varphi \left(
\left\langle x_{0},x_{0}\right\rangle b\right) +\varphi \left( b\right)
^{\ast }\varphi \left( \left\langle x_{0},x_{0}\right\rangle \right) \varphi
\left( b\right) \\
&=&\varphi \left( b^{\ast }\right) \varphi \left( \left\langle
x_{0},x_{0}\right\rangle \right) \varphi \left( b\right) -\varphi \left(
b^{\ast }\right) \varphi \left( \left\langle x_{0},x_{0}\right\rangle
\right) \varphi \left( b\right) \\
&&-\varphi \left( b\right) ^{\ast }\varphi \left( \left\langle
x_{0},x_{0}\right\rangle \right) \varphi \left( b\right) +\varphi \left(
b\right) ^{\ast }\varphi \left( \left\langle x_{0},x_{0}\right\rangle
\right) \varphi \left( b\right) \\
&=&0
\end{eqnarray*}%
whence we deduce that $\Phi \left( x_{0}b\right) =\Phi \left( x_{0}\right)
\varphi \left( b\right) $. Therefore $x_{0}\in X_{\Phi }.$

$(2)$ Let $\Phi \left( x_{0}\left\langle y,z\right\rangle \right) =\Phi
\left( x_{0}\right) \varphi \left( \left\langle y,z\right\rangle \right) $,
for all $y,z\in X$. Hence for every $y,z\in X$, 
\begin{equation*}
\Phi (x_{0})^{\ast }\Phi \left( x_{0}\left\langle y,z\right\rangle \right)
=\Phi (x_{0})^{\ast }\Phi \left( x_{0}\right) \varphi \left( \left\langle
y,z\right\rangle \right) ,
\end{equation*}%
and so $\varphi \left( \left\langle x_{0},x_{0}\right\rangle \left\langle
y,z\right\rangle \right) =\varphi \left( \left\langle
x_{0},x_{0}\right\rangle \right) \varphi \left( \left\langle
y,z\right\rangle \right) .$ Then $\left\langle x_{0},x_{0}\right\rangle \in
M_{\left. \varphi \right\vert _{I}}$, where the ideal $I$ is the closed
linear span of $\{\left\langle y,z\right\rangle :y,z\in X\}$ and so $%
\left\langle x_{0},x_{0}\right\rangle \in M_{\varphi }$, by the Lemma \ref%
{idmdlem}. Moreover, for all $b,c\in A$, we have%
\begin{eqnarray*}
\varphi (c\langle x_{0},x_{0}\rangle ^{2}b) &=&\varphi (\langle x_{0}c^{\ast
},x_{0}\langle x_{0},x_{0}b\rangle \rangle ) \\
&=&\Phi (x_{0}c^{\ast })^{\ast }\Phi (x_{0}\langle x_{0},x_{0}b\rangle
)=\Phi (x_{0}c^{\ast })^{\ast }\Phi (x_{0})\varphi (\langle
x_{0},x_{0}b\rangle ) \\
&=&\varphi (\langle x_{0}c^{\ast },x_{0}\rangle )\varphi (\langle
x_{0},x_{0}b\rangle )=\varphi (c\langle x_{0},x_{0}\rangle )\varphi (\langle
x_{0},x_{0}\rangle b) \\
&=&\varphi (c)\varphi (\langle x_{0},x_{0}\rangle )\varphi (\langle
x_{0},x_{0}\rangle )\varphi (b)=\varphi (c)\varphi (\langle x,x\rangle
^{2})\varphi (b).
\end{eqnarray*}%
This means that $\langle x_{0},x_{0}\rangle ^{2}\in T_{\varphi }$ and so $%
\langle x_{0},x_{0}\rangle \in T_{\varphi }$. Then $x_{0}\in X_{\Phi }$, by
part $(1)$.

$(3)$ Let $x\in X_{\Phi }$. Thus, for every $b\in A$, we have $\Phi \left(
xb\right) =\Phi \left( x\right) \varphi \left( b\right) $. Now, assume that $%
y$ and $z$ are two arbitrary elements in $X$. Then, we have 
\begin{eqnarray*}
\Phi (y\langle x,z\rangle )^{\ast }\Phi (y\langle x,z\rangle ) &=&\varphi
(\langle y\langle x,z\rangle ,y\langle x,z\rangle \rangle )=\varphi (\langle
z,x\rangle \langle y,y\rangle \langle x,z\rangle ) \\
&=&\varphi (\langle z,x\langle y,y\rangle \langle x,z\rangle \rangle )=\Phi
(z)^{\ast }\Phi (x\langle y,y\rangle \langle x,z\rangle ) \\
&=&\Phi (z)^{\ast }\Phi (x)\varphi (\langle y,y\rangle \langle x,z\rangle
)=\Phi (z)^{\ast }\Phi (x)\varphi (\langle x\langle y,y\rangle ,z\rangle ) \\
&=&\Phi (z)^{\ast }\Phi (x)\Phi (x\langle y,y\rangle )^{\ast }\Phi (z)=\Phi
(z)^{\ast }\Phi (x)\varphi (\langle y,y\rangle )\Phi (x)^{\ast }\Phi (z) \\
&=&\Phi (z)^{\ast }\Phi (x)\Phi (y)^{\ast }\Phi (y)\Phi (x)^{\ast }\Phi (z).
\end{eqnarray*}%
Also, 
\begin{eqnarray*}
\Phi (y)^{\ast }\Phi (y\langle x,z\rangle ) &=&\varphi (\langle y,y\langle
x,z\rangle \rangle )=\varphi (\langle x\langle y,y\rangle ,z\rangle )=\Phi
(x\langle y,y\rangle )^{\ast }\Phi (z) \\
&=&\varphi (\langle y,y\rangle )\Phi (x)^{\ast }\Phi (z)=\Phi (y)^{\ast
}\Phi (y)\Phi (x)^{\ast }\Phi (z).
\end{eqnarray*}%
These equalities imply that 
\begin{equation*}
(\Phi (y\langle x,z\rangle )-\Phi (y)\Phi (x)^{\ast }\Phi (z))^{\ast }(\Phi
(y\langle x,z\rangle )-\Phi (y)\Phi (x)^{\ast }\Phi (z))=0
\end{equation*}%
whence we deduce that $\Phi (y\langle x,z\rangle )=\Phi (y)\Phi (x)^{\ast
}\Phi (z)$. Therefore $x\in T_{\Phi }$.

Conversely, assume $x\in T_{\Phi }$, first of all we show that $\langle
x,x\rangle \in M_{\varphi }$. By the definition, $\Phi (y\langle x,z\rangle
)=\Phi (y)\Phi (x)^{\ast }\Phi (z)$, for every $y,z\in X$. Therefore%
\begin{eqnarray*}
\varphi (\langle z,y\rangle \langle x,x\rangle ) &=&\varphi (\langle
z,y\langle x,x\rangle \rangle )=\Phi (z)^{\ast }\Phi (y\langle x,x\rangle
)=\Phi (z)^{\ast }\Phi (y)\Phi (x)^{\ast }\Phi (x) \\
&=&\varphi (\langle z,y\rangle )\varphi (\langle x,x\rangle )
\end{eqnarray*}%
for every $y,z\in X$. Then $\left\langle x,x\right\rangle \in M_{\left.
\varphi \right\vert _{I}}$, where the ideal $I$ is the closed linear span of 
$\{\left\langle y,z\right\rangle :y,z\in X\}$ and so $\left\langle
x,x\right\rangle \in M_{\varphi }$, by the Lemma \ref{idmdlem}. Now, for all 
$b,c\in A$, we have%
\begin{eqnarray*}
\varphi (c\langle x,x\rangle ^{2}b) &=&\varphi (\langle xc^{\ast },x\langle
x,xb\rangle \rangle ) \\
&=&\Phi (xc^{\ast })^{\ast }\Phi (x\langle x,xb\rangle )=\Phi (xc^{\ast
})^{\ast }\Phi (x)\Phi (x)^{\ast }\Phi (xb) \\
&=&\varphi (\langle xc^{\ast },x\rangle )\varphi (\langle x,xb\rangle
)=\varphi (c\langle x,x\rangle )\varphi (\langle x,x\rangle b) \\
&=&\varphi (c)\varphi (\langle x,x\rangle )\varphi (\langle x,x\rangle
)\varphi (b)=\varphi (c)\varphi (\langle x,x\rangle ^{2})\varphi (b).
\end{eqnarray*}%
This means that $\langle x,x\rangle ^{2}\in T_{\varphi }$ and so $\langle
x,x\rangle \in T_{\varphi }$. Then $x\in X_{\Phi }$, by part $(3)$.

$(4)$ If $x_{0}\in X$ and $\Phi \left( x_{0}\left\langle
x_{0},x_{0}\right\rangle \right) =\Phi \left( x_{0}\right) \varphi \left(
\left\langle x_{0},x_{0}\right\rangle \right) $, then 
\begin{eqnarray*}
\varphi \left( \left\langle x_{0},x_{0}\right\rangle \left\langle
x_{0},x_{0}\right\rangle \right) &=&\varphi \left( \left\langle
x_{0},x_{0}\left\langle x_{0},x_{0}\right\rangle \right\rangle \right) \\
&=&\Phi \left( x_{0}\right) ^{\ast }\Phi \left( x_{0}\left\langle
x_{0},x_{0}\right\rangle \right) =\Phi \left( x_{0}\right) ^{\ast }\Phi
\left( x_{0}\right) \varphi \left( \left\langle x_{0},x_{0}\right\rangle
\right) \\
&=&\varphi \left( \left\langle x_{0},x_{0}\right\rangle \right) \varphi
\left( \left\langle x_{0},x_{0}\right\rangle \right)
\end{eqnarray*}%
and so the condition $\parallel\varphi\parallel\leq 1$ implies that $%
\left\langle x_{0},x_{0}\right\rangle \in M_{\varphi }$.

Conversely, if $\left\langle x_{0},x_{0}\right\rangle \in M_{\varphi }$ then 
\begin{eqnarray*}
&&\left( \Phi \left( x_{0}\left\langle x_{0},x_{0}\right\rangle \right)
-\Phi \left( x_{0}\right) \varphi \left( \left\langle
x_{0},x_{0}\right\rangle \right) \right) ^{\ast }\left( \Phi \left(
x_{0}\left\langle x_{0},x_{0}\right\rangle \right) -\Phi \left( x_{0}\right)
\varphi \left( \left\langle x_{0},x_{0}\right\rangle \right) \right) \\
&=&\varphi \left( \left\langle x_{0},x_{0}\right\rangle ^{3}\right) -\varphi
\left( \left\langle x_{0},x_{0}\right\rangle ^{2}\right) \varphi \left(
\left\langle x_{0},x_{0}\right\rangle \right) -\varphi \left( \left\langle
x_{0},x_{0}\right\rangle \right) \varphi \left( \left\langle
x_{0},x_{0}\right\rangle ^{2}\right) +\varphi \left( \left\langle
x_{0},x_{0}\right\rangle \right) ^{3} \\
&=&0.
\end{eqnarray*}%
Hence $\Phi \left( x_{0}\left\langle x_{0},x_{0}\right\rangle \right) =\Phi
\left( x_{0}\right) \varphi \left( \left\langle x_{0},x_{0}\right\rangle
\right) $.
\end{proof}


\begin{example}
\label{counterexample} The map $\Phi :M_{2}(\mathbb{C})\rightarrow L(\mathbb{%
C}^{2},\mathbb{C}^{2})$ defined by 
\begin{equation*}
\Phi \left( \left[ 
\begin{array}{cc}
x_{11} & x_{12} \\ 
x_{21} & x_{22}%
\end{array}%
\right] \right) =\left[ 
\begin{array}{cc}
\sqrt{2}x_{11} & 0 \\ 
\sqrt{2}x_{21} & 0%
\end{array}%
\right] 
\end{equation*}%
is a $\varphi $-map, where $\varphi :M_{2}(\mathbb{C})\rightarrow M_{2}(%
\mathbb{C})\ $is given by 
\begin{equation*}
\varphi \left( \left[ 
\begin{array}{cc}
a_{11} & a_{12} \\ 
a_{21} & a_{22}%
\end{array}%
\right] \right) =\left[ 
\begin{array}{cc}
2a_{11} & 0 \\ 
0 & 0%
\end{array}%
\right] .
\end{equation*}%
It is easy to check that $\varphi $ is a completely positive linear map
which is not contractive, and if $x_{0}=\left[ 
\begin{array}{cc}
1 & 1 \\ 
0 & 0%
\end{array}%
\right] $, then $\Phi \left( x_{0}\left\langle x_{0},x_{0}\right\rangle
\right) =\Phi \left( x_{0}\right) \varphi \left( \left\langle
x_{0},x_{0}\right\rangle \right) $, but $\left\langle
x_{0},x_{0}\right\rangle \notin M_{\varphi }$. Therefore, the contractivity
of $\varphi $ in Proposition \ref{characterization} (4) can not be neglected.
\end{example}


\begin{corollary}
For every $\varphi$-map $\Phi$, we have 
\begin{eqnarray*}
X_{\Phi }=T_{\Phi }&=&\{y\in X:\Phi \left( x\left\langle y,z\right\rangle
\right) =\Phi \left( x\right) \Phi \left( y\right) ^{\ast }\Phi \left(
z\right) \text{ for all }x,z\in X\} \\
&=& \{x\in X:\Phi \left( x\left\langle y,z\right\rangle \right) =\Phi \left(
x\right) \Phi \left( y\right) ^{\ast }\Phi \left( z\right) \text{ for all }%
y,z\in X\}
\end{eqnarray*}
\end{corollary}

\ Let $a\in M_{\varphi }$. By the polarization identity, we have $\varphi
(cab)=\varphi (c)\varphi (a)\varphi (b)$ for all $b,c\in A$ if and only if $%
\varphi (b^{\ast }ab)=\varphi (b^{\ast })\varphi (a)\varphi (b)$ for all $%
b\in A$. Therefore, 
\begin{equation*}
T_{\varphi }=\{a\in M_{\varphi }:\varphi (b^{\ast }ab)=\varphi (b^{\ast
})\varphi (a)\varphi (b)\text{ for all }b\in A\}.
\end{equation*}
This fact and the previous proposition imply the following corollary.

\begin{corollary}
\label{cor} Let $X$ be a Hilbert $A$-module, $\varphi :A\rightarrow L(%
\mathcal{H})$ a completely positive linear map, $\Phi :X\rightarrow L(%
\mathcal{H},\mathcal{K})$ a $\varphi $-map and $x_{0}\in X$. Then $x_{0}\in
X_{\Phi }$ if and only if $\left\langle x_{0},x_{0}\right\rangle \in
M_{\varphi }$ and $\varphi \left( b^{\ast }\left\langle
x_{0},x_{0}\right\rangle b\right) =\varphi \left( b^{\ast }\right) \varphi
\left( \left\langle x_{0},x_{0}\right\rangle \right) \varphi \left( b\right) 
$ for all $b\in A.$
\end{corollary}

The above results show that $\varphi$-module domains of $\varphi$-maps are
determined by $\varphi$ not by $\varphi$-maps. Indeed, Proposition \ref%
{characterization} and Corollary \ref{cor} say that the structure of $%
X_{\Phi }$ does not depend on the formula of $\Phi $ but it is completely
determined by $\varphi $. To be more precise, we have the following theorem: 

\begin{theorem}
\label{theorem37} Let $X$ be a Hilbert $A$-module and $\varphi :A\rightarrow
L(\mathcal{H})$ a completely positive linear map. If $\ \Phi
_{1}:X\rightarrow L(\mathcal{H},\mathcal{K}_{1})$ and $\Phi
_{2}:X\rightarrow L(\mathcal{H},\mathcal{K}_{2})$ are two $\varphi $-maps,
then $X_{\Phi _{1}}=X_{\Phi _{2}}$.
\end{theorem}

\begin{proof}
Let $\Phi_1$ and $\Phi_2$ be two operator-valued $\varphi$-maps. Proposition %
\ref{characterization} implies that $x\in X_{\Phi_1}$ if and only if $%
\langle x,x\rangle\in T_\varphi$ if and only if $x\in X_{\Phi_2}$.
\end{proof}

Then, we can introduce the following definition. 

\begin{definition}
\label{ternary}Let $X$ be a Hilbert $A$-module and $\varphi :A\rightarrow L(%
\mathcal{H})$ a completely positive linear map. We denote the $\varphi $%
-module domain of each $\varphi $-map on $X$ by $X_{\varphi }$ and call it
the \textbf{ternary domain of }$\varphi $ \textbf{on } $X$.
\end{definition}

We note that for every Hilbert $C^{\ast }$-module $X$ over a $C^{\ast }$%
-algebra $A$ and every completely positive linear map $\varphi :A\rightarrow
L(\mathcal{H})$, there exists a $\varphi $-map on $X$ \cite[Lemma 2.2]{ABN2}%
. Hence, the above definition is well-define. 

\begin{remark}
\label{help} We note that for every completely positive map $\varphi
:A\rightarrow L(\mathcal{H})$, $X_{\varphi }$ is a Hilbert $C^{\ast }$%
-module over the $C^{\ast }$-algebra $T_{\varphi }$ and every $\varphi $-map 
$\Phi :X\rightarrow L(\mathcal{H},\mathcal{K})$, is a ternary map on $%
X_{\varphi }$.

Also, by Proposition \ref{characterization} and Corollary \ref{cor}, we have 
\begin{eqnarray*}
X_{\varphi } &=&\{x\in X\mid \langle x,x\rangle \in M_{\varphi },\hspace{1mm}%
\varphi (a\langle x,x\rangle b)=\varphi (a)\varphi (\langle x,x\rangle
)\varphi (b)\hspace{1mm}\text{for all }a,b\in A\} \\
&=&\{x\in X\mid \langle x,x\rangle \in M_{\varphi },\hspace{1mm}\varphi
(b^{\ast }\langle x,x\rangle b)=\varphi (b)^{\ast }\varphi (\langle
x,x\rangle )\varphi (b)\hspace{1mm}\text{for all }b\in A\}.
\end{eqnarray*}%
Moreover, for every $a,b\in A$, and for every $x,y\in X_{\varphi }$ we have 
\begin{equation*}
\varphi \left( a\left\langle x,y\right\rangle b\right) =\varphi \left(
a\right) \varphi \left( \left\langle x,y\right\rangle \right) \varphi \left(
b\right) .
\end{equation*}%
The above relations explain why we use the name \textquotedblleft\ ternary
domain" for $X_{\varphi }.$
\end{remark}


\begin{proposition}
\label{tttheorem} Let $X$ be a Hilbert $A$-module, $\varphi :A\rightarrow L(%
\mathcal{H})$ a completely positive linear map and $\Phi :X\rightarrow L(%
\mathcal{H},\mathcal{K})$ a $\varphi $-map. Then the following statements
hold:

\begin{enumerate}
\item $XT_{\varphi }=X_{\varphi };$

\item $\Phi (xa)=\Phi (x)\varphi (a)$, for every $a\in M_{\varphi }$ and $%
x\in X;$

\item $\Phi(xab)=\Phi(x)\varphi(a)\varphi(b)$, for every $a\in T_\varphi$, $%
x\in X$ and $b\in A$.
\end{enumerate}
\end{proposition}

\begin{proof}
$(1)$ Let $a\in T_{\varphi }$ and $x\in X$. Since $\left\langle
xa,xa\right\rangle =a^{\ast }\left\langle x,x\right\rangle a$ and $%
T_{\varphi }AT_{\varphi }=T_{\varphi },$ $\left\langle xa,xa\right\rangle $
is an element in $T_{\varphi },$ and then, by Proposition \ref%
{characterization}, $xa\in X_{\varphi }$. Thus $XT_{\varphi }\subseteq
X_{\varphi }$. On the other hand, $X_{\varphi }$ is a Hilbert $C^{\ast }$%
-module over $T_{\varphi }$ and hence $X_{\varphi }=X_{\varphi }T_{\varphi
}\subseteq XT_{\varphi }$. Thus we showed that $XT_{\varphi }=X_{\varphi }$.

$(2)$ Let $a\in M_\varphi$ and $x\in X$. Then we have 
\begin{eqnarray*}
&&\left( \Phi \left( x a\right) -\Phi \left( x \right) \varphi \left( a
\right) \right) ^{\ast }\left( \Phi \left( x a \right) -\Phi \left( x
\right) \varphi \left( a \right) \right) \\
&=&\left( \Phi \left( x a \right) ^{\ast }-\varphi \left( a \right) ^{\ast
}\Phi \left( x \right) ^{\ast }\right) \left( \Phi \left( x a \right) -\Phi
\left( x \right) \varphi \left( a \right) \right) \\
&=&\varphi \left( a^{\ast }\left\langle x , x \right\rangle a \right)
-\varphi \left( a^{\ast }\left\langle x , x \right\rangle \right) \varphi
\left( a \right) -\varphi \left( a \right) ^{\ast }\varphi \left(
\left\langle x , x \right\rangle a \right) +\varphi \left( a \right) ^{\ast
}\varphi \left( \left\langle x , x \right\rangle \right) \varphi \left( a
\right) \\
&=& \varphi \left( a^{\ast }\right) \varphi \left( \left\langle x , x
\right\rangle \right) \varphi \left( a \right) -\varphi \left( a^{\ast
}\right) \varphi \left( \left\langle x , x \right\rangle \right) \varphi
\left( a \right) \\
&&-\varphi \left( a \right) ^{\ast }\varphi \left( \left\langle x , x
\right\rangle \right) \varphi \left( a \right) +\varphi \left( a \right)
^{\ast }\varphi \left( \left\langle x , x \right\rangle \right) \varphi
\left( a\right) \\
&=&0
\end{eqnarray*}%
whence we deduce that $\Phi \left( x a \right) =\Phi \left( x \right)
\varphi \left( a \right) $.

$(3)$ This part follows from $(1)$ and $(2)$.
\end{proof}


\begin{corollary}
\label{full} If the Hilbert $C^{\ast }$-module $X$ is full, then the Hilbert 
$C^{\ast }$-module $X_{\varphi }$ over $T_{\varphi }$ is full.
\end{corollary}


\begin{proposition}
Let $X$ be a full Hilbert $A$-module and $\varphi :A\rightarrow L(\mathcal{H}%
)$ be a completely positive linear map. Then the following statements are
equivalent:

\begin{enumerate}
\item $\varphi $ is a $\ast $-homomorphism;

\item $X_{\varphi }=X;$

\item every $\varphi $-map on $X$ is a ternary map;

\item there is a $\varphi $-map on $X$ that is a ternary map.\newline
Moreover, if $A$ is unital and $X$ has a unitary vector $x_{0}$ (i.e., $%
\left\langle x_{0},x_{0}\right\rangle =1_{A}$), then the above statements
are equivalent to

\item $x_{0}\in X_{\varphi }$.
\end{enumerate}
\end{proposition}

\begin{proof}
Obviously, $\varphi $ is a $*$-homomorphism if and only if $T_{\varphi }=A$.
Also when $A$ is unital, $T_{\varphi }=A$ if and only if $1_{A} \in
T_{\varphi }$.
\end{proof}


\section{Ternary domains and Dilation}

It is well known that the multiplicative domain of a completely positive
linear map can be characterized in terms of its minimal Stinespring dilation
triple (Theorem \ref{mtch0} ). In this section we use this results to
provide some similar results about the ternary domain of a completely
positive map on a Hilbert $C^{\ast }$-module. 

\begin{theorem}
\label{mtch0} \cite{P,G} Let $\varphi :A\rightarrow L(\mathcal{H})$ be a
completely positive linear map, and $\left( \pi _{\varphi },\mathcal{H}%
_{\varphi },V_{\varphi }\right) $ a minimal Stinespring representation
associated with $\varphi $. Then:

\begin{enumerate}
\item $a\in M_{\varphi }$ iff $V_{\varphi }^{\ast }\pi _{\varphi }(a)($id$_{%
\mathcal{H}_{\varphi }}-V_{\varphi }V_{\varphi }^{\ast })=0$ and $($id$_{%
\mathcal{H}_{\varphi }}-V_{\varphi }V_{\varphi }^{\ast })\pi _{\varphi
}(a)V_{\varphi }=0$.

\item $a\in M_{\varphi }$ iff $V_{\varphi }^{\ast }\pi _{\varphi
}(a)=\varphi (a)V_{\varphi }^{\ast }$ and $\pi _{\varphi }(a)V_{\varphi
}=V_{\varphi }\varphi (a)$.

\item $V_{\varphi }V_{\varphi }^{\ast }\in \pi _{\varphi }(M_{\varphi
})^{\prime }$, $V_{\varphi }^{\ast }V_{\varphi }\in \varphi (M_{\varphi
})^{\prime }$, $V_{\varphi }^{\ast }V_{\varphi }$ acts on $[\varphi
(M_{\varphi })\mathcal{H}]$ identically and $V_{\varphi }V_{\varphi }^{\ast
} $ acts on $[\pi _{\varphi }(M_{\varphi })V_{\varphi }\mathcal{H}]$
identically.

\item Let $B\subseteq A$, $B^{\ast }=B.\ $If $V_{\varphi }V_{\varphi }^{\ast
}$ acts on $[\pi _{\varphi }(B)V_{\varphi }\mathcal{H}]$ identically, then $%
B\subseteq M_{\varphi }$.
\end{enumerate}
\end{theorem}


\begin{proposition}
\label{mtch} Let $\varphi :A\rightarrow L(\mathcal{H})$ be a completely
positive linear map, and $\left( \pi _{\varphi },\mathcal{H}_{\varphi
},V_{\varphi }\right) $ a minimal Stinespring representation associated with 
$\varphi $. Then:

\begin{enumerate}
\item $a\in T_{\varphi }$ if and only if $a\in M_{\varphi }$ and $\pi
_{\varphi }(a)=V_{\varphi }\varphi (a)V_{\varphi }^{\ast }.$

\item If $A$ has a unit and $\varphi $ is unital then, $a\in T_{\varphi }$
iff $\pi _{\varphi }(a)=V_{\varphi }\varphi (a)V_{\varphi }^{\ast }$.

\item $V_{\varphi }V_{\varphi }^{\ast }$ acts on $[\pi _{\varphi
}(T_{\varphi })\mathcal{H}_{\varphi }]$ identically, and if $B\subseteq A$, $%
B^{\ast }=B$ and $V_{\varphi }V_{\varphi }^{\ast }$ acts on $[\pi _{\varphi
}(B)\mathcal{H}_{\varphi }]$ identically, then $B\subseteq T_{\varphi }$
\end{enumerate}
\end{proposition}

\begin{proof}
For $a,b,c\in A$, we have 
\begin{eqnarray*}
\varphi (bac)-\varphi (b)\varphi (a)\varphi (c) &=&V_{\varphi }^{\ast }\pi
_{\varphi }(b)\pi _{\varphi }(a)\pi _{\varphi }(c)V_{\varphi }-V_{\varphi
}^{\ast }\pi _{\varphi }(b)V_{\varphi }\varphi (a)V_{\varphi }^{\ast }\pi
_{\varphi }(c)V_{\varphi } \\
&=&V_{\varphi }^{\ast }\pi _{\varphi }(b)\left( \pi _{\varphi
}(a)-V_{\varphi }\varphi (a)V_{\varphi }^{\ast }\right) \pi _{\varphi
}(c)V_{\varphi }.
\end{eqnarray*}%
From the above relations and taking into account that $\left[ \pi _{\varphi
}(A)V_{\varphi }\mathcal{H}\right] =\mathcal{H}_{\varphi }$, we deduce that $%
\varphi (bac)=\varphi (b)\varphi (a)\varphi (c)$ for all $b,c\in A$ if and
only if $\pi _{\varphi }(a)=V_{\varphi }\varphi (a)V_{\varphi }^{\ast }$.

$\left( 1\right) $ According to the definition of $T_{\varphi },$ $a\in
T_{\varphi }$ if and only if $a\in M_{\varphi }$ and $\varphi (bac)=\varphi
(b)\varphi (a)\varphi (c)$ for all $b,c\in A$. Therefore, $a\in T_{\varphi }$
if and only if $a\in M_{\varphi }$ and $\pi _{\varphi }(a)=V_{\varphi
}\varphi (a)V_{\varphi }^{\ast }.$

$(2)$ If $A$ has a unit and $\varphi $ is unital, then $a\in T_{\varphi }$
if and only if $\varphi (bac)=\varphi (b)\varphi (a)\varphi (c)\ $for all $%
b,c\in A$. Therefore, $a\in T_{\varphi }$ if and only if $\pi _{\varphi
}(a)=V_{\varphi }\varphi (a)V_{\varphi }^{\ast }.$

$\left( 3\right) $ Clearly $V_{\varphi }V_{\varphi }^{\ast }$ acts on $[\pi
_{\varphi }(T_{\varphi })\mathcal{H}_{\varphi }]$ identically, since, by
part $(1)$, 
\begin{equation*}
\pi _{\varphi }(a)=V_{\varphi }\varphi (a)V_{\varphi }^{\ast }=V_{\varphi
}V_{\varphi }^{\ast }\pi _{\varphi }(a)V_{\varphi }V_{\varphi }^{\ast
}=V_{\varphi }V_{\varphi }^{\ast }\pi _{\varphi }(a)
\end{equation*}%
for all $a\in T_{\varphi }.\ $

Now, let $B$ be a self-adjoint subset of $A$ such that $V_{\varphi
}V_{\varphi }^{\ast }$ acts on $[\pi _{\varphi }(B)\mathcal{H}_{\varphi }]$
identically. Then we have $V_{\varphi }V_{\varphi }^{\ast }\pi _{\varphi
}(b)=\pi _{\varphi }(b)$ and $V_{\varphi }V_{\varphi }^{\ast }\pi _{\varphi
}(b^{\ast })=\pi _{\varphi }(b^{\ast })\ $for all $b\in B$. Thus, 
\begin{equation*}
V_{\varphi }V_{\varphi }^{\ast }\pi _{\varphi }(b)=\pi _{\varphi }(b)=\pi
_{\varphi }(b)V_{\varphi }V_{\varphi }^{\ast }
\end{equation*}%
Consequently, $V_{\varphi }\varphi (b)=\pi _{\varphi }(b)V_{\varphi }$ and $%
V_{\varphi }^{\ast }\pi _{\varphi }(b)=\varphi (b)V_{\varphi }^{\ast },$
and, by Theorem \ref{mtch0} $(2),$ $b\in M_{\varphi }$. Moreover, if $b\in B 
$, then 
\begin{equation*}
\pi _{\varphi }(b)=\pi _{\varphi }(b)V_{\varphi }V_{\varphi }^{\ast
}=V_{\varphi }V_{\varphi }^{\ast }\pi _{\varphi }(b)V_{\varphi }V_{\varphi
}^{\ast }=V_{\varphi }\varphi (b)V_{\varphi }^{\ast }
\end{equation*}%
and, by part $(1)$, $b\in T_{\varphi }$. Therefore, $B\subseteq T_{\varphi } 
$.
\end{proof}


\begin{theorem}
\label{st} Let $\varphi :A\rightarrow L(\mathcal{H})$ be a completely
positive linear map and $(\pi ,\mathcal{H}_{\varphi },V_{\varphi })$ a
minimal Stirnespring dilation triple of $\varphi $. Then for every Hilbert $%
A $-module $X$ and every $\pi $-representation $\Pi :X\rightarrow L(\mathcal{%
H}_{\varphi },\mathcal{K})$, the following statements are equivalents:

\begin{enumerate}
\item $x\in X_{\varphi };$

\item $\Pi \left( x\right) \left( \text{id}_{\mathcal{H}_{\varphi
}}-V_{\varphi }V_{\varphi }^{\ast }\right) =0$.
\end{enumerate}
\end{theorem}

\begin{proof}
Let $x\in X_{\varphi }$. By Proposition \ref{characterization}, $\langle
x,x\rangle \in T_{\varphi }$. Since $\langle x,x\rangle \in T_{\varphi
}\subseteq M_{\varphi }$, Proposition \ref{mtch} $(1,3)$ and Theorem \ref%
{mtch0} $\left( 3\right) $ implies that 
\begin{equation*}
\pi (\langle x,x\rangle )=V_{\varphi }\varphi (\langle x,x\rangle
)V_{\varphi }^{\ast }=V_{\varphi }V_{\varphi }^{\ast }\pi (\langle
x,x\rangle )V_{\varphi }V_{\varphi }^{\ast }=\pi (\langle x,x\rangle
)V_{\varphi }V_{\varphi }^{\ast }=V_{\varphi }V_{\varphi }^{\ast }\pi
(\langle x,x\rangle ).
\end{equation*}%
Therefore%
\begin{eqnarray*}
&&(\Pi (x)(\text{id}_{\mathcal{H}_{\varphi }}-V_{\varphi }V_{\varphi }^{\ast
}))^{\ast }(\Pi (x)(\text{id}_{\mathcal{H}_{\varphi }}-V_{\varphi
}V_{\varphi }^{\ast })) \\
&=&\Pi (x)^{\ast }\Pi (x)-\Pi (x)^{\ast }\Pi (x)V_{\varphi }V_{\varphi
}^{\ast }-V_{\varphi }V_{\varphi }^{\ast }\Pi (x)^{\ast }\Pi (x)+V_{\varphi
}V_{\varphi }^{\ast }\Pi (x)^{\ast }\Pi (x)V_{\varphi }V_{\varphi }^{\ast }
\\
&=&\pi \left( \left\langle x,x\right\rangle \right) -\pi \left( \left\langle
x,x\right\rangle \right) V_{\varphi }V_{\varphi }^{\ast }-V_{\varphi
}V_{\varphi }^{\ast }\pi \left( \left\langle x,x\right\rangle \right)
+V_{\varphi }V_{\varphi }^{\ast }\pi \left( \left\langle x,x\right\rangle
\right) V_{\varphi }V_{\varphi }^{\ast } \\
&=&0
\end{eqnarray*}

Thus, $\Pi (x)(\text{id}_{{\mathcal{H}_{\varphi }}}-V_{\varphi }V_{\varphi
}^{\ast })=0$.

Conversely, assume $x\in X$ and $\Pi (x)(\text{id}_{\mathcal{H}_{\varphi
}}-V_{\varphi }V_{\varphi }^{\ast })=0$. Then 
\begin{equation*}
0=\Pi (x)^{\ast }\Pi (x)(\text{id}_{\mathcal{H}_{\varphi }}-V_{\varphi
}V_{\varphi }^{\ast })=\pi \left( \left\langle x,x\right\rangle \right)
\left( \text{id}_{\mathcal{H}_{\varphi }}-V_{\varphi }V_{\varphi }^{\ast
}\right) .
\end{equation*}%
Therefore, $\pi (\langle x,x\rangle )=\pi (\langle x,x\rangle )V_{\varphi
}V_{\varphi }^{\ast }$. From this relation and taking into account that $\pi
(\langle x,x\rangle )$ is positive, we deduce that 
\begin{equation*}
V_{\varphi }V_{\varphi }^{\ast }\pi (\langle x,x\rangle )=\pi (\langle
x,x\rangle )V_{\varphi }V_{\varphi }^{\ast }=\pi (\langle x,x\rangle ).
\end{equation*}%
Therefore by Proposition \ref{mtch} part (3), $\langle x,x\rangle \in
T_{\varphi }$, which is equivalent to $x\in X_{\varphi }$ ( see Proposition %
\ref{characterization} $(2)$ and Definition \ref{ternary} ).
\end{proof}

It has been shown that every $\varphi $-map on a Hilbert $A$-module $X$ can
be dilated to every $\pi $-representation on $X$, where $\pi $ is a dilation
of the completely positive map $\varphi $ \cite[Theorem 2.3]{ABN2}. We
recall this fact and some related results as a remark and we use it to prove
another characterization of the ternary domain $X_{\varphi }$, in terms of
the dilation of the $\varphi $-map, similar to Proposition \ref{mtch}.

\begin{remark}
\label{ABN1} Let $X$ be a Hilbert $A$-module, $\varphi :A\rightarrow L(%
\mathcal{H})$ a completely positive linear map and $(\pi ,\mathcal{H}%
_\varphi,V_\varphi)$ be a minimal Stinespring dilation triple of $\varphi $.

\begin{enumerate}
\item There exists a triple $((\Phi _{\varphi },\mathcal{K}_{\varphi }),(\Pi
_{\pi },\mathcal{K}_{\pi }),W_{\varphi })$ consisting of two Hilbert spaces $%
\mathcal{K}_{\varphi }$ and $\mathcal{K}_{\pi },$ a $\varphi $-map $\Phi
_{\varphi }:X\rightarrow L(\mathcal{H},\mathcal{K}_{\varphi })$, a $\pi $%
-map $\Pi _{\pi }:X\rightarrow L(\mathcal{H}_{\varphi },\mathcal{K}_{\pi })$
and a unitary operator $W_{\varphi }:\mathcal{K}_{\varphi }\rightarrow 
\mathcal{K}_{\pi },$ such that 
\begin{equation*}
\Phi _{\varphi }(\cdot )=W_{\varphi }^{\ast }\Pi _{\pi }(\cdot )V_{\varphi }.
\end{equation*}%
Moreover, $[\Phi _{\varphi }(X)\mathcal{H}]=\mathcal{K}_{\varphi }$ and $%
[\Pi _{\pi }(X)\mathcal{H}_{\varphi }]=\mathcal{K}_{\pi }$ \emph{\cite[Lemma
2.2]{ABN2}}.

\item If $\Phi :X\rightarrow L(\mathcal{H},\mathcal{K})$ is a $\varphi $%
-map, there exists a unique isometry $S_{\Phi }:\mathcal{K}_{\varphi
}\rightarrow \mathcal{K}$ such that $S_{\Phi }\Phi _{\varphi
}(\cdot)=\Phi(\cdot) $. Moreover, if $[\Phi (X)\mathcal{H}]=\mathcal{K}$,
then $S_{\Phi }$ is a unitary operator \emph{\cite[Theorem 2.3]{ABN2}}.

\item If $\Phi :X\rightarrow L(\mathcal{H},\mathcal{K})$ is a $\varphi $%
-map, $\Pi :X\rightarrow L(\mathcal{K}_{\varphi },\widetilde{\mathcal{K}})$
is a $\pi $-map, the operator $W=S_{\Pi }W_{\varphi }S_{\Phi }^{\ast }$ is a
partial isometry and $\Phi (x)=W^{\ast }\Pi (x)V_{\varphi }$ for all $x\in X$%
.
\end{enumerate}
\end{remark}


\begin{corollary}
\label{st10} Let $\varphi :A\rightarrow L(\mathcal{H})$ be a completely
positive linear map and $(\pi,\mathcal{H}_\varphi,V_\varphi)$ a minimal
Stirnespring dilation triple of $\varphi$. Then for every Hilbert $A$-module 
$X$, the following statements are equivalent:

\begin{enumerate}
\item $x\in X_{\varphi };$

\item $\pi(\langle y,x\rangle) \left( \text{id}_{\mathcal{H}%
_\varphi}-V_\varphi V_\varphi^{\ast }\right) =0$ for all $y\in X$.
\end{enumerate}
\end{corollary}


\begin{corollary}
\label{Phich} Let $\varphi :A\rightarrow L(\mathcal{H})$ be a completely
positive linear map, $(\pi ,\mathcal{H}_{\varphi },V_{\varphi })$ a minimal
Stirnespring dilation triple of $\varphi $. Then for every Hilbert $A$%
-module $X$, every $\pi $-representation $\Pi :X\rightarrow L(\mathcal{H}%
_{\varphi },\mathcal{K})$ and every $\varphi $-map $\Phi :X\rightarrow L(%
\mathcal{H},\widetilde{\mathcal{K}})$, there is a partial isometry $W$ such
that $\Phi (\cdot )=W^{\ast }\Pi (\cdot )V_{\varphi }$ and the following
statements are equivalent:

\begin{enumerate}
\item $x\in X_{\varphi };$

\item $\Pi (x)($id$_{\mathcal{H}_{\varphi }}-V_{\varphi }V_{\varphi }^{\ast
})=0;$

\item $W^{\ast }\Pi (x)=\Phi (x)V_{\varphi }^{\ast };$

\item $\Pi (x)=W\Phi (x)V_{\varphi }^{\ast }$.
\end{enumerate}
\end{corollary}

\begin{proof}
Let $\Pi :X\rightarrow L(\mathcal{H}_{\varphi },\mathcal{K})$ be a $\pi $%
-representation, $\Phi :X\rightarrow L(\mathcal{H},\widetilde{\mathcal{K}})$
a $\varphi $-map and $((\Phi _{\varphi },\mathcal{K}_{\varphi }),(\Pi _{\pi
},\mathcal{K}_{\pi }),W_{\varphi })$ be the triple as in Remark \ref{ABN1}.
As mentioned in Remark \ref{ABN1}, there exist the isometries $S_{\Phi }:%
\mathcal{K}_{\varphi }\rightarrow \widetilde{\mathcal{K}}$ and $S_{\Pi }:%
\mathcal{K}_{\pi }\rightarrow \mathcal{K}$ such that $S_{\Phi }\Phi
_{\varphi }(\cdot )=\Phi (\cdot )$ and $S_{\Pi }\Pi _{\pi }(\cdot )=\Pi
(\cdot )$. Moreover, $W=S_{\Pi }W_{\varphi }S_{\Phi }^{\ast }$ is a partial
isometry and $\Phi (\cdot )=W^{\ast }\Phi (\cdot )V_{\varphi }$.

$(1)\Longleftrightarrow (2)$ This equivalence follows from Theorem \ref{st},
since $\Pi $ is a $\pi $-representation and $(\pi ,\mathcal{H}_{\varphi
},V_{\varphi })$ is a minimal Stinespring dilation triple for $\varphi $.

$(1)\Longrightarrow (3)$ Let $x\in X_{\varphi }$. By Theorem \ref{st}, $\Pi
_{\pi }(x)($id$_{\mathcal{H}_{\varphi }}-V_{\varphi }V_{\varphi }^{\ast
})=0, $ whence we deduce that $W_{\varphi }^{\ast }\Pi _{\pi }(x)=\Phi
_{\varphi }(x)V_{\varphi }^{\ast }$. Then%
\begin{eqnarray*}
W^{\ast }\Pi (x) &=&S_{\Phi }W_{\varphi }^{\ast }S_{\Pi }^{\ast }\Pi
(x)=S_{\Phi }W_{\varphi }^{\ast }S_{\Pi }^{\ast }S_{\Pi }\Pi _{\pi }(x) \\
&=&S_{\Phi }W_{\varphi }^{\ast }\Pi _{\pi }(x)=S_{\Phi }\Phi _{\varphi
}(x)V_{\varphi }^{\ast }=\Phi (x)V_{\varphi }^{\ast }.
\end{eqnarray*}

$(3)\Longrightarrow (2)$ Assume $x\in X$ and $W^{\ast }\Pi (x)=\Phi
(x)V_{\varphi }^{\ast }$. Therefore, $S_{\Phi }W_{\varphi }^{\ast }S_{\Pi
}^{\ast }\Pi (x)=\Phi (x)V_{\varphi }^{\ast }.$ From this relation and
taking into account that $S_{\Phi }$ is an isometry and $W_{\varphi }^{\ast
} $ is a unitary we obtain 
\begin{eqnarray*}
S_{\Pi }^{\ast }\Pi (x) &=&W_{\varphi }S_{\Phi }^{\ast }\Phi (x)V_{\varphi
}^{\ast }=W_{\varphi }S_{\Phi }^{\ast }S_{\Phi }\Phi _{\varphi
}(x)V_{\varphi }^{\ast }=W_{\varphi }\Phi _{\varphi }(x)V_{\varphi }^{\ast }
\\
&=&W_{\varphi }W_{\varphi }^{\ast }\Pi _{\pi }(x)V_{\varphi }V_{\varphi
}^{\ast }=\Pi _{\pi }(x)V_{\varphi }V_{\varphi }^{\ast }=S_{\Pi }^{\ast }\Pi
(x)V_{\varphi }V_{\varphi }^{\ast }.
\end{eqnarray*}%
Therefore,%
\begin{equation*}
S_{\Pi }^{\ast }\Pi (x)(\text{id}_{\mathcal{H}_{\varphi }}-V_{\varphi
}V_{\varphi }^{\ast })=0.
\end{equation*}%
Since $S_{\Pi }$ is an isometry and $S_{\Pi }S_{\Pi }^{\ast }$ is the
orthogonal projection on $[\Pi (X)\mathcal{H}_{\varphi }],$ $\Pi (x)=S_{\Pi
}S_{\Pi }^{\ast }\Pi (x),$ and thus 
\begin{equation*}
\Pi (x)(\text{id}_{\mathcal{H}_{\varphi }}-V_{\varphi }V_{\varphi }^{\ast
})=S_{\Pi }S_{\Pi }^{\ast }\Pi (x)(\text{id}_{\mathcal{H}_{\varphi
}}-V_{\varphi }V_{\varphi }^{\ast })=0.
\end{equation*}

$(3)\Longrightarrow (4)$ Assume $x\in X$ and $W^{\ast }\Pi (x)=\Phi
(x)V_{\varphi }^{\ast }.$ Since $WW^{\ast }=S_{\Pi }S_{\Pi }^{\ast }$ and $%
S_{\Pi }S_{\Pi }^{\ast }$ is the orthogonal projection on $[\Pi (X)\mathcal{H%
}_{\varphi }]$, we have 
\begin{equation*}
\begin{split}
\Pi (x)&=S_{\Pi }S_{\Pi }^{\ast }\Pi (x)=WW^{\ast }\Pi (x)=
WS_{\Phi}W_{\varphi}^*S_{\Pi}^*\Pi(x) \\
&=WS_{\Phi}W_{\varphi}^*\Pi_{\pi}(x)= WS_{\Phi}\Phi_{\varphi}(x)V_\varphi=
W\Phi (x)V_{\varphi }^{\ast }.
\end{split}%
\end{equation*}

$(4)\Longrightarrow (3)$ Assume $x\in X$ and $\Pi (x)=W\Phi (x)V_{\varphi
}^{\ast }$. Since $S_{\Pi }$ is an isometry, $W^{\ast }W=S_{\Phi }W_{\varphi
}^{\ast }W_{\varphi }S_{\Phi }^{\ast }.$ On the other hand, $W_{\varphi
}^{\ast }W_{\varphi }=\text{id}_{\mathcal{K}_{\varphi }} $ and then $W^{\ast
}W=S_{\Phi }S_{\Phi }^{\ast },$ which is the orthogonal projection on $[\Phi
(X)\mathcal{H}]$. Thus 
\begin{equation*}
W^{\ast }\Pi (x)=W^{\ast }W\Phi (x)V_{\varphi }^{\ast }=S_{\Phi }S_{\Phi
}^{\ast }\Phi (x)V_{\varphi }^{\ast }=\Phi (x)V_{\varphi }^{\ast }.
\end{equation*}
\end{proof}

As an immediate consequence, we have the following characterization. 

\begin{corollary}
\label{stch} Let $\varphi :A\rightarrow L(\mathcal{H})$ be a completely
positive linear map, $\Phi :X\rightarrow L(\mathcal{H},\mathcal{K})$ a $%
\varphi $-map and $\left( \left( \Pi _{\Phi },\pi _{\varphi }\right)
,(W_{\Phi },V_{\varphi }),(\mathcal{H}_{\varphi },\mathcal{K}_{\Phi
})\right) $ a minimal Stinespring representation associated with $(\varphi
,\Phi )$. Then the following statements are equivalent:

\begin{enumerate}
\item $x\in X_{\varphi };$

\item $\Pi _{\Phi }\left( x\right) \left( \text{id}_{\mathcal{H}_{\varphi
}}-V_{\varphi }V_{\varphi }^{\ast }\right) =0;$

\item $W_{\Phi }^{\ast }\Pi _{\Phi }(x)=\Phi (x)V_{\varphi }^{\ast };$

\item $\Pi _{\Phi }(x)=W_{\Phi }\Phi (x)V_{\varphi }^{\ast }.$
\end{enumerate}
\end{corollary}


\section{Multiplicative domain of an induced completely positive map on
linking $C^{\ast }$-algebra}

In this section, we show that a $\varphi $-map $\Phi $ on a Hilbert $A$%
-module $X$ induces a unique completely positive linear map on $\mathcal{L}%
(X)$, the linking $C^{\ast }$-algebra of $X$. Furthermore, we show that
every completely positive linear map $\varphi :A\rightarrow L(\mathcal{H})$
induces a unique (in a some sense) completely positive linear map on $%
\mathcal{L}(X)$. Moreover, we determine the multiplicative domain of the
induced completely positive linear maps on $\mathcal{L}(X)$ in terms of the
multiplicative domain $M_{\varphi }$ and the ternary domain $X_{\varphi }$.
Finally, we show that all the completely positive linear maps on $\mathcal{L}%
(X)$ which are induced by a completely positive linear map $\varphi
:A\rightarrow L(\mathcal{H})$ have the same multiplicative domain. 

\begin{proposition}
\label{pro111} Let $X$ be a Hilbert $A$-module, $\varphi :A\rightarrow L(%
\mathcal{H})$ a completely positive linear map and $\Phi :X\rightarrow L(%
\mathcal{H},\mathcal{K})$ a $\varphi $-map. Then there is a unique $\ast $%
-representation $\phi _{\Phi ,\varphi }:K(X)\rightarrow L(\mathcal{K})$ such
that $\widetilde{\varphi }_{\Phi }=%
\begin{bmatrix}
\phi _{\Phi ,\varphi } & \Phi \\ 
\Phi ^{\ast } & \varphi%
\end{bmatrix}%
:\mathcal{L}(X)\rightarrow L(\mathcal{K}\oplus \mathcal{H})$ is a completely
positive linear map and for every minimal Stinespring representation
associated to the $\varphi $-map $\Phi $ such as $\left( \left( \Pi ,\pi
\right) ,\left( W,V\right) ,\left( \mathcal{K}^{\prime },\mathcal{H}^{\prime
}\right) \right) $, there is a $\ast $ -representation $\Gamma
:K(X)\rightarrow L(\mathcal{K}^{\prime })$ such that $\Gamma (\theta
_{x,y})=\Pi (x)\Pi (y)^{\ast }$ for all $x,y\in X$ and 
\begin{equation*}
\widetilde{\varphi }_{\Phi }=%
\begin{bmatrix}
\phi _{\Phi ,\varphi } & \Phi \\ 
\Phi ^{\ast } & \varphi%
\end{bmatrix}%
=%
\begin{bmatrix}
W^{\ast } & 0 \\ 
0 & V^{\ast }%
\end{bmatrix}%
\begin{bmatrix}
\Gamma & \Pi \\ 
\Pi ^{\ast } & \pi%
\end{bmatrix}%
\begin{bmatrix}
W & 0 \\ 
0 & V%
\end{bmatrix}%
.
\end{equation*}%
Moreover, if $\varphi $\ is contractive, then $\widetilde{\varphi }_{\Phi }$%
\ is contractive.
\end{proposition}

\begin{proof}
Let $\left( \left( \Pi _{\Phi },\pi _{\varphi }\right) ,(W_{\Phi
},V_{\varphi }),(\mathcal{H}_{\varphi },\mathcal{K}_{\Phi })\right) $ be a
minimal Stinespring representation for $\varphi $-map $\Phi \ $(see \cite[%
Theorem 2.1]{BRS}).

Then 
\begin{equation*}
\Pi _{\Phi ,\varphi }=%
\begin{bmatrix}
\Gamma _{\Phi ,\varphi } & \Pi _{\Phi } \\ 
\Pi _{\Phi }^{\ast } & \pi _{\varphi }%
\end{bmatrix}%
\end{equation*}%
where $\Gamma _{\Phi ,\varphi }:K(X)\rightarrow L(\mathcal{K}_{\Phi })$ is a 
$\ast $ -representation of $K(X)$ on $\mathcal{K}_{\Phi }$ such that $\Gamma
_{\Phi ,\varphi }(\theta _{x,y})=\Pi _{\Phi }(x)\Pi _{\Phi }(y)^{\ast }$ for
all $x,y\in X$, is a $\ast $-representation of $\mathcal{L}(X)$ on $\mathcal{%
K}_{\Phi }\oplus \mathcal{H}_{\varphi }$ (see, for example, \cite[p.13]{Ar}%
). We consider the linear map $\widetilde{\varphi }_{\Phi }:\mathcal{L}%
(X)\rightarrow L(\mathcal{K}\oplus \mathcal{H})$ given by 
\begin{equation*}
\widetilde{\varphi }_{\Phi }=%
\begin{bmatrix}
W_{\Phi }^{\ast } & 0 \\ 
0 & V_{\varphi }^{\ast }%
\end{bmatrix}%
\begin{bmatrix}
\Gamma _{\Phi ,\varphi } & \Pi _{\Phi } \\ 
\Pi _{\Phi }^{\ast } & \pi _{\varphi }%
\end{bmatrix}%
\begin{bmatrix}
W_{\Phi } & 0 \\ 
0 & V_{\varphi }%
\end{bmatrix}%
.
\end{equation*}%
Clearly, $\widetilde{\varphi }_{\Phi }$ is a completely positive linear map
from $\mathcal{L}(X)$ to $L(\mathcal{K}\oplus \mathcal{H})$, and if $\varphi 
$\ is contractive, then $\widetilde{\varphi }_{\Phi }$\ is contractive. It
is easy to check that 
\begin{equation*}
\widetilde{\varphi }_{\Phi }=%
\begin{bmatrix}
\phi _{\Phi ,\varphi } & \Phi \\ 
\Phi ^{\ast } & \varphi%
\end{bmatrix}%
\end{equation*}%
where $\phi _{\Phi ,\varphi }\left( \cdot \right) =W_{\Phi }^{\ast }\Gamma
_{\Phi ,\varphi }\left( \cdot \right) W_{\Phi }$. Since $W_{\Phi }$ is a
coisometry, $\phi _{\Phi ,\varphi }:K(X)\rightarrow L(\mathcal{K})$ is a $%
C^{\ast }$-morphism.

Let $\left( \left( \Pi ^{\prime },\pi ^{\prime }\right) ,\left( W^{\prime
},V^{\prime }\right) ,\left( \mathcal{K}^{\prime },\mathcal{H}^{\prime
}\right) \right) $ be another minimal Stinespring representation associated
to the $\varphi $-map $\Phi $, and $\varphi _{\Phi }^{\prime }=%
\begin{bmatrix}
\phi _{\Phi ,\varphi }^{\prime } & \Phi \\ 
\Phi ^{\ast } & \varphi%
\end{bmatrix}%
:\mathcal{L}(X)\rightarrow L(\mathcal{K}\oplus \mathcal{H})$ be the
completely positive linear map associated to this representation.

By \cite[Theorem 2.4]{BRS}, there exist two unitary operators $U_{1}:%
\mathcal{H}_{\varphi }\rightarrow \mathcal{H}^{\prime }$ and $U_{2}:\mathcal{%
K}_{\Phi }\rightarrow \mathcal{K}^{\prime }$ such that $U_{1}V_{\varphi
}=V^{\prime }$, $U_{1}\pi _{\varphi }(\cdot )=\pi ^{\prime }(\cdot )U_{1}$, $%
U_{2}W_{\Phi }=W^{\prime }$ and $U_{2}\Pi _{\Phi }(\cdot )=\Pi ^{\prime
}(\cdot )U_{1}$. Then, 
\begin{eqnarray*}
\phi _{\Phi ,\varphi }^{\prime }(\theta _{x,y}) &=&W^{\prime \ast }\Pi
^{\prime }(x)\Pi ^{\prime }(y)^{\ast }W^{\prime }=W_{\Phi }^{\ast
}U_{2}^{\ast }U_{2}\Pi _{\Phi }(x)U_{1}^{\ast }U_{1}\Pi _{\Phi }(y)^{\ast
}U_{2}^{\ast }U_{2}W_{\Phi } \\
&=&W_{\Phi }^{\ast }\Pi _{\Phi }(x)\Pi _{\Phi }(y)^{\ast }W_{\Phi }=\phi
_{\Phi ,\varphi }(\theta _{x,y})
\end{eqnarray*}%
for all $x,y\in X$. Therefore, $\phi _{\Phi ,\varphi }^{\prime }=\phi _{\Phi
,\varphi }$ and then $\widetilde{\varphi }_{\Phi }=\varphi _{\Phi }^{\prime
}.$
\end{proof}

The following lemma will be used in the main theorem of this section. Also,
it expose some important features of ternary domains. Indeed in this lemma,
we show that the ternary domain of a completely positive map on a Hilbert $%
C^{\ast }$-module is a left Hilbert $C^{\ast }$-module over the set of
compact operators on the Hilbert $C^{\ast }$-module, i.e., $X_{\varphi }$ is
a $K(X)$-$M_{\varphi }$-bimodule. Moreover, we provide another
characterization of the ternary domain in term of a specific representation
of compact operators. 

\begin{lemma}
\label{lem42} Let $X$ be a Hilbert $A$-module, $\varphi :A\rightarrow L(%
\mathcal{H})$ a completely positive linear map and $\Phi :X\rightarrow L(%
\mathcal{H},\mathcal{K})$ a $\varphi $-map. Then we have:

\begin{enumerate}
\item For every $T\in K(X)$ and $z\in X$, $\Phi(T(z))=\phi_{\Phi,\varphi}(T)%
\Phi(z)$;

\item $K(X)X_\varphi\subseteq X_\varphi$; \newline
Moreover, if $\varphi$ is a contraction, then

\item $x\in X_\varphi$ if and only if $\phi_{\Phi,\varphi}(\theta_{x,x})=%
\Phi(x)\Phi(x)^*$;

\item $x\in X_\varphi$ if and only if $\phi_{\Phi,\varphi}(\theta_{x,y})=%
\Phi(x)\Phi(y)^*$ for all $y\in X$;

\item $\phi_{\Phi,\varphi}(\theta_{xa,y})=\Phi(xa)\Phi(y)^*=\Phi(x)%
\varphi(a)\Phi(y)^*$ for all $x,y\in X$ and $a\in T_\varphi$.
\end{enumerate}
\end{lemma}

\begin{proof}
Let $\left( \left( \Pi _{\Phi },\pi _{\varphi }\right) ,(W_{\Phi
},V_{\varphi }),(\mathcal{H}_{\varphi },\mathcal{K}_{\Phi })\right) $ be a
minimal Stinespring representation for $\varphi $-map $\Phi \ $(see \cite[%
Theorem 2.1]{BRS}).

$(1)$ For every $x,y,z\in X$, we have%
\begin{eqnarray*}
\Phi (\theta _{x,y}(z)) &=&\Phi (x\langle y,z\rangle )=W_{\Phi }^{\ast }\Pi
_{\Phi }(x\langle y,z\rangle )V_{\varphi } \\
&=&W_{\Phi }^{\ast }\Pi _{\Phi }(x)\Pi _{\Phi }(y)^{\ast }\Pi _{\Phi
}(z)V_{\varphi } \\
&=&W_{\Phi }^{\ast }\Pi _{\Phi }(x)\Pi _{\Phi }(y)^{\ast }W_{\Phi }W_{\Phi
}^{\ast }\Pi _{\Phi }(z)V_{\varphi } \\
&=&\phi _{\Phi ,\varphi }(\theta _{x,y})\Phi (z).
\end{eqnarray*}%
Since the linear span generated by $\{\theta _{x,y}:x,y\in X\}$ is a dense
subset of $K(X)$ and the left module action of $K(X)$ on $X$ is continuous, $%
\Phi (T(z))=\phi _{\Phi ,\varphi }(T)\Phi (z)$ for every $T\in K(X)$ and $%
z\in X$.

$(2)$ Let $x,y\in X$ and $z\in X_{\varphi }$, 
\begin{eqnarray*}
\Phi \left( \theta _{x,y}\left( z\right) a\right) &=&W_{\Phi }^{\ast }\Pi
_{\Phi }\left( x\left\langle y,z\right\rangle a\right) V_{\varphi }=W_{\Phi
}^{\ast }\Pi _{\Phi }\left( x\right) \pi _{\varphi }\left( \left\langle
y,za\right\rangle \right) V_{\varphi } \\
&=&W_{\Phi }^{\ast }\Pi _{\varphi }\left( x\right) \Pi _{\Phi }\left(
y\right) ^{\ast }\Pi _{\Phi }\left( za\right) V_{\varphi } \\
&=&W_{\Phi }^{\ast }\Pi _{\Phi }\left( x\right) \Pi _{\Phi }\left( y\right)
^{\ast }W_{\Phi }W_{\Phi }^{\ast }\Pi _{\Phi }\left( za\right) V_{\varphi }
\\
&=&\phi _{\Phi ,\varphi }\left( \theta _{x,,y}\right) \Phi \left( za\right)
=\phi _{\Phi ,\varphi }\left( \theta _{x,,y}\right) \Phi \left( z\right)
\varphi \left( a\right) \\
&=&\Phi \left( \theta _{x,y}\left( z\right) \right) \varphi \left( a\right)
\end{eqnarray*}%
for all $a\in A,$ and so $\theta _{x,y}\left( z\right) \in X_{\varphi }.$%
Since the linear span generated by $\{\theta _{x,y}:x,y\in X\}$ is a dense
subset of $K(X)$, we conclude that $K(X)X_{\varphi }\subseteq X_{\varphi }.$ 
\newline

From now on, we assume that $\varphi$ is a contraction. \newline
$(3)$ Let $x\in X_{\varphi }$. Then, by Corollary \ref{stch} $(3)$, $W_{\Phi
}^{\ast }\Pi _{\Phi }(x)=\Phi (x)V_{\varphi }^{\ast },$ and 
\begin{equation*}
\phi _{\Phi ,\varphi }(\theta _{x,x})=W_{\Phi }^{\ast }\Pi _{\Phi }(x)\Pi
_{\Phi }(x)^{\ast }W_{\Phi }=\Phi (x)V_{\varphi }^{\ast }\Pi _{\Phi
}(x)^{\ast }W_{\Phi }=\Phi (x)\Phi (x)^{\ast }.
\end{equation*}%
Conversely, assume $x\in X$ such that $\phi _{\Phi ,\varphi }(\theta
_{x,x})=\Phi (x)\Phi (x)^{\ast }$. Then%
\begin{eqnarray*}
\widetilde{\varphi }_{\Phi }\left( 
\begin{bmatrix}
0 & x \\ 
0 & 0%
\end{bmatrix}%
\begin{bmatrix}
0 & x \\ 
0 & 0%
\end{bmatrix}%
^{\ast }\right) &=&%
\begin{bmatrix}
\phi _{\Phi ,\varphi }(\theta _{x,x}) & 0 \\ 
0 & 0%
\end{bmatrix}%
=%
\begin{bmatrix}
\Phi (x)\Phi (x)^{\ast } & 0 \\ 
0 & 0%
\end{bmatrix}
\\
&=&\widetilde{\varphi }_{\Phi }\left( 
\begin{bmatrix}
0 & x \\ 
0 & 0%
\end{bmatrix}%
\right) \widetilde{\varphi }_{\Phi }\left( 
\begin{bmatrix}
0 & x \\ 
0 & 0%
\end{bmatrix}%
\right) ^{\ast }
\end{eqnarray*}%
and by \cite[Theorem 3.18 (ii)]{P}, 
\begin{equation*}
\widetilde{\varphi }_{\Phi }\left( 
\begin{bmatrix}
0 & x \\ 
0 & 0%
\end{bmatrix}%
\begin{bmatrix}
0 & 0 \\ 
0 & a%
\end{bmatrix}%
\right) =\widetilde{\varphi }_{\Phi }\left( 
\begin{bmatrix}
0 & x \\ 
0 & 0%
\end{bmatrix}%
\right) \widetilde{\varphi }_{\Phi }\left( 
\begin{bmatrix}
0 & 0 \\ 
0 & a%
\end{bmatrix}%
\right)
\end{equation*}%
for all $a\in A$. Consequently, $\Phi (xa)=\Phi (x)\varphi (a)$ for all $%
a\in A$, and thus $x\in X_{\varphi }$.

$(4)$ If $x\in X$ and $\phi _{\Phi ,\varphi }(\theta _{x,y})=\Phi (x)\Phi
(y)^{\ast }$ for all $y\in X$, then $\phi _{\Phi ,\varphi }(\theta
_{x,x})=\Phi (x)\Phi (x)^{\ast }$, and, by $(3)$, $x\in X_{\varphi }.$

Conversely, assume that $x\in X_{\varphi }$ and $y\in X$. By Corollary \ref%
{stch} $(3)$,%
\begin{equation*}
\phi _{\Phi ,\varphi }(\theta _{x,y})=W_{\Phi }^{\ast }\Pi _{\Phi }(x)\Pi
_{\Phi }(y)^*W_{\Phi }=\Phi (x)V_{\varphi }^{\ast }\Pi _{\Phi }(y)^*W_{\Phi
}=\Phi (x)\Phi (y)^{\ast }.
\end{equation*}

$(5)$ By Proposition \ref{tttheorem}, for every $x\in X$ and $a\in
T_{\varphi }$, $xa\in X_{\varphi }$ and $\Phi (xa)=\Phi (x)\varphi (a)$.
Using $(4),$ we obtain 
\begin{equation*}
\phi _{\Phi ,\varphi }(\theta _{xa,y})=\Phi (xa)\Phi (y)^{\ast }=\Phi
(x)\varphi (a)\Phi (y)^{\ast }.
\end{equation*}
\end{proof}


The above lemma says that, when we consider $X$ as a left Hilbert $K(X)$%
-module then the given $\varphi $-map $\Phi $ is a $\phi _{\Phi ,\varphi }$%
-module map, but it is not a $\phi _{\Phi ,\varphi }$-representation, in
general. In fact, $\Phi $ is a $\phi _{\Phi ,\varphi }$-representation if
and only if $X_{\varphi }=X$. However, $X_{\varphi }$ is a right (left)
Hilbert $C^{\ast } $-module over $M_{\varphi }$ (respectively $K(X)$) such
that $\Phi |_{X_{\varphi }}$, the restriction of $\Phi $ to $X_{\varphi }$,
is both a $\varphi \vert_{M_{\varphi }} $-representation and a $\phi _{\Phi
,\varphi }$-representation. Surprisingly, $X_{\varphi }$ is the largest
Hilbert $C^{\ast }$-module in $X$ with this property. In fact, we say more
as follows. 

\begin{corollary}
Let $X$ be a Hilbert $A$-module, $\varphi :A\rightarrow L(\mathcal{H})$ a
contractive completely positive linear map and $\Phi :X\rightarrow L(%
\mathcal{H},\mathcal{K})$ a $\varphi $-map. Also, assume that $Y\ \
(Y\subseteq X)$ is a Hilbert $C^{\ast }$-module over a $C^{\ast }$%
-subalgebra $B$ of $A$ (which inherits the operations of $X$) such that $%
\Phi $ acts on $Y$ as a ternary map.

\begin{enumerate}
\item If $Y$ is full (as a Hilbert $B$-module), then $B \subseteq M_{\varphi
}$ and $Y = Y_{\varphi \vert_B}$;

\item If $B=A$, then $Y\subseteq X_{\varphi }$;

\item If $[\Phi (Y)\mathcal{H}]=\mathcal{K}$, then $Y\subseteq X_{\varphi }$;

\item $Y\subseteq X_{\varphi }$ if and only if $\left. \Phi \right\vert _{Y}$
is a $\left. \phi _{\Phi ,\varphi }\right\vert _{K(Y)}$-representation.
\end{enumerate}
\end{corollary}

\begin{proof}
$(1)$ Let $x \in Y$. We have 
\begin{equation*}
\Phi \left( x\left\langle x , x\right\rangle \right) =\Phi(x) \Phi(x)^*
\Phi(x)=\Phi \left( x \right) \varphi \left( \left\langle x , x
\right\rangle \right).
\end{equation*}
Hence by Proposition \ref{characterization}, $\left\langle x , x
\right\rangle \in M_\varphi$. Consequently, by the polarization identity, $%
\left\langle v , w \right\rangle \in M_\varphi$, for all $v, w \in Y$. Hence 
$B \subseteq M_{\varphi }$, whenever $Y$ is full.

$(2)$ Let $x \in Y$. It is well-known that there are elements $u, v, w$ in $%
Y $ such that $x=u \left\langle v , w \right\rangle$. Hence for every $a \in
A$, we have 
\begin{eqnarray*}
\Phi(x a) &=& \Phi(u \left\langle v , w \right\rangle a) = \Phi(u
\left\langle v , w a \right\rangle )= \Phi(u) \Phi(v)^* \Phi(wa) \\
&=& \Phi(u) \varphi(\left\langle v , w a \right\rangle)= \Phi(u)
\varphi(\left\langle v , w \right\rangle) \varphi(a)= \Phi(x)\varphi(a).
\end{eqnarray*}
Therefore $x \in X_{\Phi }= X_{\varphi }$.

$(3)$ By Lemma \ref{lem42}, for every $x,y\in Y$ we have $\Phi (\theta
_{x,x}(y))=\phi _{\Phi ,\varphi }(\theta _{x,x})\Phi (y)$. Thus for every $%
x,y\in Y$, we have 
\begin{equation*}
\Phi (x)\Phi (x)^{\ast }\Phi (y)=\Phi (x\langle x,y\rangle )=\Phi (\theta
_{x,x}(y))=\phi _{\Phi ,\varphi }(\theta _{x,x})\Phi (y).
\end{equation*}%
Now, if $[\Phi (Y)\mathcal{H}]=\mathcal{K}$, then the above equalities imply
that $\phi _{\Phi ,\varphi }(\theta _{x,x})=\Phi (x)\Phi (x)^{\ast }$ for
all $x\in Y$. Hence $Y\subseteq X_{\varphi }$, by Lemma \ref{lem42}.

$(4)$ By Lemma \ref{lem42}, $Y\subseteq X_{\varphi }$ if and only if $%
\phi_{\Phi,\varphi}(\theta_{x,y})=\Phi(x)\Phi(y)^*$, for all $x, y \in Y$.
This means that $\Phi\vert_Y$ is a $\left. \phi _{\Phi ,\varphi
}\right\vert_{K(Y)}$-representation.
\end{proof}

In the following proposition, we investigate how a completely positive
linear map $\varphi :A\rightarrow L(\mathcal{H})$ induces a completely
positive linear map on the linking $C^{\ast }$-algebra $\mathcal{L}(X)$. 

\begin{proposition}
\label{pro112} Let $X$ be a Hilbert $A$-module and $\varphi :A\rightarrow L(%
\mathcal{H})$ a completely positive linear map. Then there is a $\ast $%
-representation $\phi _{\varphi }:K(X)\rightarrow $ $L(\mathcal{K}_{\varphi
})$ and a $\varphi$-map $\Phi_\varphi:X\rightarrow L(\mathcal{H},\mathcal{%
K_\varphi})$ such that

\begin{enumerate}
\item $[\Phi _{\varphi }(X)\mathcal{H}]=\mathcal{K}_{\varphi }$ and $%
\widetilde{\varphi }_{\Phi _{\varphi }}:=%
\begin{bmatrix}
\phi _{\varphi } & \Phi _{\varphi } \\ 
\Phi _{\varphi }^{\ast } & \varphi%
\end{bmatrix}%
$ is a completely positive linear map on $\mathcal{L}(X)$;

\item For every $\varphi $-map $\Phi :X\rightarrow L(\mathcal{H},\mathcal{K}%
) $ there is an isometry $\ S_{\Phi }:\mathcal{K}_{\varphi }\rightarrow 
\mathcal{K}$ such that $\phi _{\Phi ,\varphi }\left( \cdot \right) =S_{\Phi
}\phi _{\varphi }\left( \cdot \right) S_{\Phi }^{\ast }$. Also, if $[\Phi (X)%
\mathcal{H}]=\mathcal{K}$, then $\ S_{\Phi }\ $is a unitary operator. 
\newline
Moreover, if $\varphi $ is pure, then:

\item $\phi _{\varphi }$ is an irreducible $\ast $-representation;

\item $\widetilde{\varphi }_{\Phi _{\varphi }}$ is pure;

\item The isometry $S_{\Phi }$ is unique up to a complex number of module
one.
\end{enumerate}
\end{proposition}

\begin{proof}
$(1)$ Let $(\pi ,\mathcal{H}_{\varphi },V_{\varphi })$ be a minimal
Stinespring representation associated to $\varphi $. Then there exists a
triple $((\Phi _{\varphi },\mathcal{K}_{\varphi }),(\Pi _{\pi },\mathcal{K}%
_{\pi }),W_{\varphi })$ consisting of two Hilbert spaces $\mathcal{K}%
_{\varphi }$ and $\mathcal{K}_{\pi }$, a unitary operator $W_{\varphi }:%
\mathcal{K}_{\varphi }\rightarrow \mathcal{K}_{\pi }$, a $\varphi $-map $%
\Phi _{\varphi }:X\rightarrow L(\mathcal{H},\mathcal{K}_{\varphi })$ and a $%
\pi $-representation $\Pi _{\pi }:X\rightarrow L(\mathcal{H}_{\varphi },%
\mathcal{K}_{\pi })$ such that $\Phi _{\varphi }(\cdot )=W_{\varphi }^{\ast
}\Pi _{\pi }(\cdot )V_{\varphi }$, $[\Phi _{\varphi }(X)\mathcal{H}]=%
\mathcal{K}_{\varphi }$ and $[\Pi _{\pi }(X)\mathcal{H}_{\varphi }]=\mathcal{%
K}_{\pi }$ (see Remark \ref{ABN1} ). Clearly, $\left( \left( \Pi _{\pi },\pi
\right) ,\left( \mathcal{H}_{\varphi },\mathcal{K}_{\pi }\right) ,\left(
V_{\varphi },W_{\varphi }\right) \right) $ is a minimal Stinespring
representation associated to the $\varphi $-map $\Phi _{\varphi }$. By
Proposition \ref{pro111}, there is a $\ast $-representation $\phi _{\varphi
}:K(X)\rightarrow L(\mathcal{H}_{\varphi })\ $ such that 
\begin{equation*}
\phi _{\varphi }\left( \theta _{x,y}\right) =W_{\varphi }^{\ast }\Pi _{\pi
}\left( x\right) \Pi _{\pi }\left( y\right) ^{\ast }W_{\varphi },
\end{equation*}%
and the map $\widetilde{\varphi }_{\Phi _{\varphi }}$ is a completely
positive linear map on $\mathcal{L}(X)$.

$(2)$ Suppose that $\Phi :X\rightarrow L(\mathcal{H},\mathcal{K})$ is
another $\varphi $-map. Then, by Remark \ref{ABN1} $(2)$, there is a unique
isometry $S_{\Phi }:\mathcal{K}_{\varphi }\rightarrow \mathcal{K}$ such that 
$S_{\Phi }\Phi _{\varphi }(\cdot )=\Phi (\cdot )$. Since $\left( \left( \Psi
_{\pi },\pi \right) ,\left( \mathcal{H}_{\varphi },\mathcal{K}_{\pi }\right)
,\left( V_{\varphi },W_{\varphi }S_{\Phi }^{\ast }\right) \right) $ is a
minimal Stinespring representation for the $\varphi $-map $\Phi $,$\ $by
Proposition \ref{pro111}, for the $\ast $-representation $\phi _{\Phi
,\varphi }:K(X)\rightarrow \mathcal{L(K)},\ $ we have 
\begin{equation*}
\phi _{\Phi ,\varphi }\left( \theta _{x,y}\right) =S_{\Phi }W_{\varphi
}^{\ast }\Pi _{\pi }\left( x\right) \Pi _{\pi }\left( y\right) ^{\ast
}W_{\varphi }S_{\Phi }^{\ast }=S_{\Phi }\phi _{\varphi }\left( \theta
_{x,y}\right) S_{\Phi }^{\ast }.
\end{equation*}%
Moreover, if $[\Phi (X)\mathcal{H}]=\mathcal{K}$, then $S_{\Phi }$ is a
unitary operator\ (see Remark \ref{ABN1} $(2))$.

$(3)$ If $\varphi $ is pure, since $(\pi ,\mathcal{H}_{\varphi },V_{\varphi
})$ is a minimal Stinespring representation associated to $\varphi ,$ the $%
\ast $-representation $\pi $ is irreducible. On the other hand, since $[\Pi
_{\pi }(X)\mathcal{H}_{\varphi }]=\mathcal{K}_{\pi }$, by \cite[Proposition
3.6 and Lemma 3.5]{Ar}\textbf{, }the $\ast $ -representations $\Pi _{\pi }$ $%
:X\rightarrow $ $L(\mathcal{H}_{\varphi },\mathcal{K}_{\pi })$, $\Gamma
_{\pi }:$ $K(X)\rightarrow L(\mathcal{K}_{\pi }),$ $\Gamma _{\pi }\left(
\theta _{x,y}\right) =$ $\Pi _{\pi }(x)\Pi _{\pi }(y)^{\ast }$ and $%
\begin{bmatrix}
\Gamma _{\pi } & \Pi _{\pi } \\ 
\Pi _{\pi }^{\ast } & \pi%
\end{bmatrix}%
:\mathcal{L}(X)\rightarrow L(\mathcal{K}_{\pi }\oplus \mathcal{H}_{\varphi
}) $ are irreducible. Since \textbf{\ }$W_{\varphi }$ is a unitary operator,
the $\ast $-representations $\Gamma _{\pi }$ and $\phi _{\varphi }$ are
unitarily equivalent, and so $\phi _{\varphi }$ is irreducible.

$\left( 4\right) $ Since the triple $\left( 
\begin{bmatrix}
\Gamma _{\pi } & \Pi _{\pi } \\ 
\Pi _{\pi }^{\ast } & \pi%
\end{bmatrix}%
,\mathcal{K}_{\pi }\oplus \mathcal{H}_{\varphi },%
\begin{bmatrix}
W_{\varphi } & 0 \\ 
0 & V_{\varphi }%
\end{bmatrix}%
\right) $ is an irreducible dilation for $\widetilde{\varphi }_{\Phi
_{\varphi }}$, $\widetilde{\varphi }_{\Phi _{\varphi }}$ is pure.

$(5)$ Let $S:\mathcal{K}_{\varphi }\rightarrow \mathcal{K}$ be another
isometry such that $\phi _{\Phi ,\varphi }\left( \cdot \right) =S\phi
_{\varphi }\left( \cdot \right) S^{\ast }$. Then 
\begin{equation*}
S^{\ast }S_{\Phi }\phi _{\varphi }\left( \cdot \right) =S^{\ast }S_{\Phi
}\phi _{\varphi }\left( \cdot \right) S_{\Phi }^{\ast }S_{\Phi }=S^{\ast
}\phi _{\Phi ,\varphi }\left( \cdot \right) S_{\Phi }=S^{\ast }S\phi
_{\varphi }\left( \cdot \right) S^{\ast }S_{\Phi }=\phi _{\varphi }\left(
\cdot \right) S^{\ast }S_{\Phi }
\end{equation*}%
and so $S^{\ast }S_{\Phi }$ is an element in the commutant of $\phi
_{\varphi }\left( K(X)\right) $. Since $\phi _{\varphi }$ is irreducible,
there is a complex number $\alpha $ such that $S^{\ast }S_{\Phi }=\alpha $id$%
_{\mathcal{K}_{\varphi }}$. From 
\begin{eqnarray*}
S\phi _{\varphi }\left( \theta _{x,y}\right) S^{\ast }\left( \Phi \left(
z\right) \xi \right) &=&\phi _{\Phi ,\varphi }\left( \theta _{x,y}\right)
\left( \Phi \left( z\right) \xi \right) \\
&=&S_{\Phi }W_{\varphi }^{\ast }\Psi _{\pi }\left( x\right) \Psi _{\pi
}\left( y\right) ^{\ast }W_{\varphi }S_{\Phi }^{\ast }S_{\Phi }W_{\varphi
}^{\ast }\Psi _{\pi }\left( z\right) V_{\varphi }\xi \\
&=&S_{\Phi }W_{\varphi }^{\ast }\Psi _{\pi }\left( x\right) \pi \left(
\left\langle y,z\right\rangle \right) V_{\varphi }\xi =S_{\Phi }W_{\varphi
}^{\ast }\Psi _{\pi }\left( x\left\langle y,z\right\rangle \right)
V_{\varphi }\xi \\
&=&\Phi \left( x\left\langle y,z\right\rangle \right) \xi
\end{eqnarray*}%
for all $x,y,z\in X$ and $\xi \in \mathcal{H}$ and taking into account that $%
X\left\langle X,X\right\rangle $ is dense in $X,\ $we deduce that $\left[
\Phi \left( X\right) \mathcal{H}\right] \subseteq S\left( \mathcal{K}%
_{\varphi }\right) .$

On the other hand, since $S$ is an isometry, $SS^{\ast }$ is a projection on
the range $S\left( \mathcal{K}_{\varphi }\right) $ of $S$. Thus, we have 
\begin{equation*}
\alpha S\left( \Phi _{\varphi }(x)\xi \right) =SS^{\ast }S_{\Phi }\left(
\Phi _{\varphi }(x)\xi \right) =SS^{\ast }\left( \Phi \left( x\right) \xi
\right) =\Phi \left( x\right) \xi
\end{equation*}%
for all $\xi \in \mathcal{H}$. From the above relation and Remark \ref{ABN1} 
$(2)$, we deduce that $\alpha S=S_{\Phi },\ $and since $S$ and $S_{\Phi }$
are isometries, $\left\vert \alpha \right\vert =1.$
\end{proof}


\begin{corollary}
\label{cprelation} Let $X$ be a Hilbert $A$-module and $\varphi
:A\rightarrow L(\mathcal{H})$ a completely positive linear map. Then, for
every $\varphi $-map $\Phi :X\rightarrow L(\mathcal{H},\mathcal{K})$ there
is an isometry $S_{\Phi }$ such that 
\begin{equation*}
\widetilde{\varphi }_{\Phi }\left( \cdot \right) =\text{diag}(S_{\Phi },%
\text{id}_{\mathcal{H}})\widetilde{\varphi }_{\Phi _{\varphi }}\left( \cdot
\right) \text{diag}(S_{\Phi }^{\ast },\text{id}_{\mathcal{H}}).
\end{equation*}
\end{corollary}


\begin{remark}
For a given completely positive map $\varphi :A\rightarrow L(\mathcal{H})$
and a right Hilbert $A$-module $X$, the $\varphi $-map $\Phi _{\varphi
}:X\rightarrow L(\mathcal{H},\mathcal{K}_{\varphi })$ and the $\ast $%
-representation $\phi _{\varphi }:K(X)\rightarrow $ $L(\mathcal{K}_{\varphi
})$ are unique (up to unitary equivalence). Therefore the induced completely
positive map $\widetilde{\varphi }_{\Phi _{\varphi }}$ associated with $%
\varphi $ in Proposition \ref{pro112} is unique (up to unitary equivalence).
\end{remark}

In Proposition \ref{pro111}, we showed that a $\varphi $-map $\Phi $ on a
Hilbert $C^{\ast }$-module $X$ induces a completely positive linear map $%
\widetilde{\varphi }_{\Phi }$ on the linking $C^{\ast }$-algebra $\mathcal{L}%
(X)$. In the next theorem, we determine the multiplicative (and the ternary)
domain of $\widetilde{\varphi }_{\Phi }$.

\begin{theorem}
\label{thmlink} Let $X$ be a Hilbert $A$-module, $\varphi :A\rightarrow L(%
\mathcal{H})$ a contractive completely positive linear map and $\Phi
:X\rightarrow L(\mathcal{H},\mathcal{K})$ a $\varphi $-map. If $\widetilde{%
\varphi }_{\Phi }:\mathcal{L}(X)\rightarrow L(\mathcal{K}\oplus \mathcal{H})$
is the completely positive linear map on $\mathcal{L}(X)$ associated to the $%
\varphi $-map $\Phi $, then $\mathcal{L}_{M_\varphi}(X_{\varphi })\subseteq
M_{\widetilde{\varphi }_{\Phi }}$. Moreover, 
\begin{equation*}
M_{\widetilde{\varphi }_{\Phi }}=\left\{ 
\begin{bmatrix}
T & y \\ 
x^{\ast } & a%
\end{bmatrix}%
:T\in K(X),x,y\in X_{\varphi },a\in M_{\varphi }\right\},
\end{equation*}
and also, 
\begin{equation*}
T_{\widetilde{\varphi }_{\Phi }}=\left\{ 
\begin{bmatrix}
T & y \\ 
x^{\ast } & a%
\end{bmatrix}%
:T\in K(X),x,y\in X_{\varphi },a\in T_{\varphi }\right\}.
\end{equation*}
\end{theorem}

\begin{proof}
Let $%
\begin{bmatrix}
T & y \\ 
x^{\ast } & a%
\end{bmatrix}%
\in \mathcal{L}(X)$. By \cite[Theorem 2.1]{G},$\ 
\begin{bmatrix}
T & y \\ 
x^{\ast } & a%
\end{bmatrix}%
\in M_{\widetilde{\varphi }_{\Phi }}$ if and only if 
\begin{equation}
\widetilde{\varphi }_{\Phi }\left( 
\begin{bmatrix}
T & y \\ 
x^{\ast } & a%
\end{bmatrix}%
^{\ast }%
\begin{bmatrix}
T & y \\ 
x^{\ast } & a%
\end{bmatrix}%
\right) =\widetilde{\varphi }_{\Phi }\left( 
\begin{bmatrix}
T & y \\ 
x^{\ast } & a%
\end{bmatrix}%
\right) ^{\ast }\widetilde{\varphi }_{\Phi }\left( 
\begin{bmatrix}
T & y \\ 
x^{\ast } & a%
\end{bmatrix}%
\right)  \label{eq3}
\end{equation}%
and 
\begin{equation}
\widetilde{\varphi }_{\Phi }\left( 
\begin{bmatrix}
T & y \\ 
x^{\ast } & a%
\end{bmatrix}%
\begin{bmatrix}
T & y \\ 
x^{\ast } & a%
\end{bmatrix}%
^{\ast }\right) =\widetilde{\varphi }_{\Phi }\left( 
\begin{bmatrix}
T & y \\ 
x^{\ast } & a%
\end{bmatrix}%
\right) \widetilde{\varphi }_{\Phi }\left( 
\begin{bmatrix}
T & y \\ 
x^{\ast } & a%
\end{bmatrix}%
\right) ^{\ast }.  \label{eq4}
\end{equation}

A simple calculus shows that the relation (\ref{eq3}) holds if and only if%
\begin{eqnarray*}
&&%
\begin{bmatrix}
\phi _{\Phi ,\varphi }\left( T^{\ast }T\right) +\phi _{\Phi ,\varphi }\left(
\theta _{x,x}\right) & \Phi \left( T^{\ast }\left( y\right) \right) +\Phi
\left( xa\right) \\ 
\Phi \left( T^{\ast }\left( y\right) \right) ^{\ast }+\Phi \left( xa\right)
^{\ast } & \varphi \left( \left\langle y,y\right\rangle \right) +\varphi
\left( a^{\ast }a\right)%
\end{bmatrix}
\\
&=&%
\begin{bmatrix}
\phi _{\Phi ,\varphi }\left( T^{\ast }\right) \phi _{\Phi ,\varphi }\left(
T\right) +\Phi \left( x\right) \Phi \left( x\right) ^{\ast } & \phi _{\Phi
,\varphi }\left( T^{\ast }\right) \Phi \left( y\right) +\Phi \left( x\right)
\varphi \left( a\right) \\ 
\Phi \left( y\right) ^{\ast }\phi _{\Phi ,\varphi }\left( T\right) +\varphi
\left( a\right) ^{\ast }\Phi \left( x\right) ^{\ast } & \Phi \left( y\right)
^{\ast }\Phi \left( y\right) +\varphi \left( a\right) ^{\ast }\varphi \left(
a\right)%
\end{bmatrix}%
\end{eqnarray*}%
Now, since $\phi _{\Phi ,\varphi }$ is a $*$-homomorphism, $\Phi $ is a $%
\varphi $-map and $\Phi (S(z))=\phi _{\Phi ,\varphi }(S)\Phi (z)$ for all $%
S\in K(X)$ and for all $z\in X\ $ [Lemma \ref{lem42} (1)], the relation (\ref%
{eq3}) holds if and only if 
\begin{equation*}
\left\{ 
\begin{array}{ccc}
\phi _{\Phi ,\varphi }\left( \theta _{x,x}\right) & = & \Phi \left( x\right)
\Phi \left( x\right) ^{\ast } \\ 
\Phi \left( xa\right) & = & \Phi \left( x\right) \varphi \left( a\right) \\ 
\varphi \left( a^{\ast }a\right) & = & \varphi \left( a\right) ^{\ast
}\varphi \left( a\right)%
\end{array}%
\right. .
\end{equation*}%
By Lemma \ref{lem42} (3) we conclude that the relation (\ref{eq3}) holds if
and only if $x\in X_{\varphi }$ and $\varphi (a^{\ast }a)=\varphi (a)^{\ast
}\varphi (a)$.

Similarly, the relation (\ref{eq4}) holds if and only if $y\in X_{\varphi }$
and $\varphi \left( aa^{\ast }\right) =\varphi \left( a\right) \varphi
\left( a\right) ^{\ast }$. Therefore $%
\begin{bmatrix}
T & y \\ 
x^{\ast } & a%
\end{bmatrix}%
\in M_{\widetilde{\varphi }_{\Phi }}$ if and only if $a\in M_{\varphi }$, $%
x,y\in X_{\varphi }$. So, 
\begin{equation*}
M_{\widetilde{\varphi }_{\Phi }}=\left\{ 
\begin{bmatrix}
T & y \\ 
x^{\ast } & a%
\end{bmatrix}%
:T\in K(X),x,y\in X_{\varphi },a\in M_{\varphi }\right\} .
\end{equation*}

In particular, we have 
\begin{equation*}
\mathcal{L}_{M_{\varphi }}(X_{\varphi })=\left\{ 
\begin{bmatrix}
T & y \\ 
x^{\ast } & a%
\end{bmatrix}%
:T\in K(X_{\varphi }),x,y\in X_{\varphi },a\in M_{\varphi }\right\}
\subseteq M_{\widetilde{\varphi }_{\Phi }}.
\end{equation*}

Finally, assume that $%
\begin{bmatrix}
T & x \\ 
y^{\ast } & a%
\end{bmatrix}%
\in T_{\widetilde{\varphi }_{\Phi }}$. Since $T_{\widetilde{\varphi }_{\Phi
}}\subseteq M_{\widetilde{\varphi }_{\Phi }}$, then $x,y\in X_{\varphi }$, $%
a\in M_{\varphi }$ and $T\in K(X)$. Moreover, we have 
\begin{eqnarray*}
&&\widetilde{\varphi }_{\Phi }\left( 
\begin{bmatrix}
0 & 0 \\ 
0 & b%
\end{bmatrix}%
^{\ast }%
\begin{bmatrix}
T & x \\ 
y^{\ast } & a%
\end{bmatrix}%
\begin{bmatrix}
0 & 0 \\ 
0 & b%
\end{bmatrix}%
\right) \\
&=&\widetilde{\varphi }_{\Phi }\left( 
\begin{bmatrix}
0 & 0 \\ 
0 & b%
\end{bmatrix}%
\right) ^{\ast }\widetilde{\varphi }_{\Phi }\left( 
\begin{bmatrix}
T & x \\ 
y^{\ast } & a%
\end{bmatrix}%
\right) \widetilde{\varphi }_{\Phi }\left( 
\begin{bmatrix}
0 & 0 \\ 
0 & b%
\end{bmatrix}%
\right)
\end{eqnarray*}%
for all $b\in A.$ This equality implies that $\varphi \left( b^{\ast
}ab\right) =\varphi \left( b\right) ^{\ast }\varphi \left( a\right) \varphi
\left( a\right) \ $for all $b\in A,\ $and so $a\in T_{\varphi }$.

Conversely, let $x,y\in X_{\varphi }$, $a\in T_{\varphi }$ and $T\in K(X)$.
A simple calculation shows that 
\begin{eqnarray*}
&&\widetilde{\varphi }_{\Phi }\left( 
\begin{bmatrix}
R & u \\ 
v^{\ast } & b%
\end{bmatrix}%
^{\ast }%
\begin{bmatrix}
0 & x \\ 
y^{\ast } & 0%
\end{bmatrix}%
\begin{bmatrix}
R & u \\ 
v^{\ast } & b%
\end{bmatrix}%
\right) \\
&=&\widetilde{\varphi }_{\Phi }\left( 
\begin{bmatrix}
R & u \\ 
v^{\ast } & b%
\end{bmatrix}%
\right) ^{\ast }\widetilde{\varphi }_{\Phi }\left( 
\begin{bmatrix}
0 & x \\ 
y^{\ast } & 0%
\end{bmatrix}%
\right) \widetilde{\varphi }_{\Phi }(\left( 
\begin{bmatrix}
R & u \\ 
v^{\ast } & b%
\end{bmatrix}%
\right)
\end{eqnarray*}%
for all $%
\begin{bmatrix}
R & u \\ 
v^{\ast } & b%
\end{bmatrix}%
\in \mathcal{L}(X)$. Therefore, $%
\begin{bmatrix}
0 & x \\ 
y^{\ast } & 0%
\end{bmatrix}%
\in T_{\widetilde{\varphi }_{\Phi }}$. Also, by using Proposition \ref%
{tttheorem} and Lemma \ref{lem42}, one can show that $%
\begin{bmatrix}
T & 0 \\ 
0 & a%
\end{bmatrix}%
\in T_{\widetilde{\varphi }_{\Phi }}$. Therefore $%
\begin{bmatrix}
T & x \\ 
y^{\ast } & a%
\end{bmatrix}%
\in T_{\widetilde{\varphi }_{\Phi }}$.
\end{proof}


\begin{remark}
We note that in the above theorem, the contractivity assumption is not
necessary. In fact, the above theorem remains valid for all completely
positive linear map $\varphi:A\rightarrow L(\mathcal{H})$.
\end{remark}

Consequently, the completely positive linear maps on the linking $C^{\ast }$%
-algebra $\mathcal{L}(X)$ which are induced by $\varphi $-maps have the same
multiplicative domain and ternary domain.

\begin{corollary}
Let $\varphi :A\rightarrow L(\mathcal{H})$ be a completely positive linear
map and $\Phi $ and $\Psi $ be two operator-valued $\varphi $-maps on a
Hilbert $A$-module $X$. Then $M_{\widetilde{\varphi }_{\Phi }}=M_{\widetilde{%
\varphi }_{\Psi }}$, and $T_{\widetilde{\varphi }_{\Phi }}=T_{\widetilde{%
\varphi }_{\Psi }}$.
\end{corollary}

\subsection*{Acknowledgment}

The research of the first author was in part supported by a grant from IPM
(No. 96470123).

%
%

\bigskip

\bigskip

\bigskip

\bigskip

\end{document}